\documentclass{article}
\usepackage[utf8]{inputenc}

\usepackage[latin9]{luainputenc}

 \usepackage{geometry}
\usepackage{babel}
\usepackage[unicode=true,pdfusetitle, bookmarks=true,bookmarksnumbered=false,bookmarksopen=false, breaklinks=false,pdfborder={0 0 1},backref=false,colorlinks=false] {hyperref}

\usepackage{amsmath,amssymb,amsthm,enumitem}
\usepackage{stmaryrd}

\usepackage[pdftex]{graphicx}
\usepackage[T1]{fontenc}
\usepackage{enumitem}

\theoremstyle{definition}
\newtheorem{thm}{Theorem}[section]
\newtheorem{lem}[thm]{Lemma}
\newtheorem{prop}[thm]{Proposition}
\newtheorem{cor}[thm]{Corollary}
\newtheorem{rmk}[thm]{Remark}
\newtheorem{defn}[thm]{Definition}

\newcommand{\pt}{\partial_t}

\newcommand{\intT}{\int_{\mathbb{T}}}

\newcommand{\R}{\mathbb{R}}
\newcommand{\T}{\mathbb{T}}
\newcommand{\N}{\mathbb{N}}
\newcommand{\ue}{u^{\epsilon}}
\newcommand{\me}{m^{\epsilon}}

\title{Regularity and long time behavior of one-dimensional first-order mean field games and the planning problem}
\author{Nikiforos Mimikos-Stamatopoulos\thanks{University of Chicago, Department of Mathematics, Chicago, IL, \texttt{nmimikos@math.uchicago.edu}} \and Sebastian Muñoz\thanks{University of Chicago, Department of Mathematics, Chicago, IL,  \texttt{sbstn@math.uchicago.edu}}}
\date{}

\begin{document}

\maketitle
\begin{abstract}
     We study the regularity and long time behavior of the one-dimensional, local, first-order mean field games system and the planning problem, assuming a Hamiltonian of superlinear growth, with a non-separated, strictly monotone dependence on the density. We improve upon the existing literature by obtaining two regularity results. The first is the existence of classical solutions without the need to assume blow-up of the cost function near small densities. The second result is the interior smoothness of weak solutions without the need to assume neither blow-up of the cost function nor that the initial density be bounded away from zero. We also characterize the long time behavior of the solutions, proving that they satisfy the turnpike property with an exponential rate of convergence, and that they converge to the solution of the infinite horizon system. Our approach relies on the elliptic structure of the system and displacement convexity estimates. In particular, we apply displacement convexity methods to obtain both global and local a priori lower bounds on the density.
\end{abstract}
\textit{\small{}MSC: 35Q89, 35B65, 35J66, 35J70. }{\small\par}
\noindent \textit{\small{}Keywords: displacement convexity; Hamilton-Jacobi equations; quasilinear elliptic equations; oblique derivative problems; Bernstein method; non-linear method of continuity.}{\small\par}

\tableofcontents{}
\section{Introduction}
The main purpose of this paper is to establish that, under very general conditions, the solutions to the one-dimensional first-order mean field games system with local coupling (MFG for short) are smooth, and to fully characterize their long time behavior. Specifically, we study the regularity of the solutions to standard MFG with a prescribed terminal condition,

\begin{equation}\tag{MFG}
    \label{MainMFGSystem}
    \begin{cases}
        -u_t(x,t)+H(u_x(x,t),m(x,t))=0 & (x,t)\in Q_{T}=\T \times(0,T),\\
        m_t(x,t)-(m(x,t)H_p(u_x(x,t),m(x,t)))_x=0 & (x,t)\in Q_{T}, \\
        m(x,0)=m_0(x),\, u(x,T)=g(m(x,T)) & x\in\T,
    \end{cases}
\end{equation}
as well as to the so-called planning problem with a prescribed terminal density,
\begin{equation}\tag{MFGP}
    \label{PlanningSystem}
    \begin{cases}
        -u_t(x,t)+H(u_x(x,t),m(x,t))=0 & (x,t)\in Q_{T},\\
        m_t(x,t)-(m(x,t)H_p(u_x(x,t),m(x,t)))_x=0 & (x,t)\in Q_{T}, \\
        m(x,0)=m_0(x),\, m(x,T)=m_T(x) & x\in\T,
    \end{cases}
\end{equation}
where $\T$ denotes the $1$-dimensional torus,  $-H(p,m):\mathbb{R}\times(0,\infty)\rightarrow \mathbb{R}$ and $g(m):(0,\infty)\rightarrow \mathbb{R}$ are strictly increasing in $m$, $H$ has super-linear growth in $p$, and $m_0,m_T:\mathbb{T}\rightarrow [0,+\infty)$ are probability densities. We also show convergence of the solutions to each of these problems, as $T\rightarrow \infty$, to the solution of the infinite time horizon MFG system, 
\begin{equation}\tag{MFGL}
    \label{InfiniteHorizonSystem}
    \begin{cases}
        -v_t(x,t)+\lambda+H(v_x(x,t),\mu(x,t))=0 & (x,t)\in \T \times(0,\infty),\\
        \mu_t(x,t)-(\mu(x,t)H_p(v_x(x,t),\mu(x,t)))_x=0 & (x,t)\in \T \times(0,\infty), \\
        \mu(x,0)=m_0(x) & x\in\T,
    \end{cases}
\end{equation}
where $\lambda = -H(0,1)$.\\

MFG were introduced by Lasry and Lions \cite{Lions,LasryLions}, and
at the same time, in a particular setting, by Caines, Huang, and Malhamé
\cite{caines}. They are non-cooperative differential games with infinitely
many players, in which the players find an optimal strategy, determined by the value function $u$, by observing
the distribution $m$ of the other players. 

Classical solutions to \eqref{MainMFGSystem}, in arbitrary dimension, were previously obtained by the second author in \cite{Munoz1,Munoz2}, when the initial density is bounded away from 0, and under the blow-up assumption 
\begin{equation}\label{BlowUpCondition}\lim_{m\rightarrow 0^+}H(p,m)=+\infty,\end{equation}
which, from the optimal control point of view, corresponds to placing a very strong incentive for players to occupy low-density regions and precludes the appearance of empty regions. A similar regularity result was recently obtained in \cite{Porretta2} by A. Porretta for the case of (\ref{PlanningSystem}), when the Hamiltonian has the separated form $H(p,m)\equiv H(p)-f(m)$, and the terminal density $m_T$ is also bounded away from 0.

Our first contribution is the following theorem, which shows that, in the one-dimensional problem, assumption (\ref{BlowUpCondition}) can be completely removed. We refer to Section \ref{SectionAssumptions} for assumptions \hyperref[eq:M1]{(M)},
\hyperref[eq:Hpp bd]{(H)} \hyperref[g increasing]{(G)}, \hyperref[eq:ellipticity]{(E)}, \hyperref[eq:mweak]{(W)}, and \hyperref[eq:H long time]{(L)}, and to the \hyperref[SubsectionNotation]{notation} subsection for the definition of the function spaces mentioned below.

\begin{thm} \label{MainTheorem}
Let $0<\alpha<1$, and assume that \hyperref[eq:M1]{(M)},
\hyperref[eq:Hpp bd]{(H)}, \hyperref[g increasing]{(G)},
and \hyperref[eq:ellipticity]{(E)} hold. Then the following statements hold: 
\begin{enumerate}[label=(\roman*)]
   \item There exists a classical solution $(u,m)\in C^{3,\alpha}(\overline{Q_{T}})\times C^{2,\alpha}(\overline{Q_{T}})$
to (\ref{PlanningSystem}). The function $m$ is unique, and $u$ is unique up to a constant.
    \item There exists a unique classical solution $(u,m)\in C^{3,\alpha}(\overline{Q_{T}})\times C^{2,\alpha}(\overline{Q_{T}})$
to (\ref{MainMFGSystem}).
 \end{enumerate}
\end{thm}
Our second result establishes interior smoothness of the solutions when, besides removing the assumption \eqref{BlowUpCondition}, one also weakens the  lower bound assumptions for given densities $m_0$ and $m_T$, replacing the latter with the integrability conditions
\begin{equation} \label{m0mTintegrability}
    \frac{1}{m_0^{\kappa}}\in L^1(\T),\,\,\,\frac{1}{m_T^{\kappa}}\in L^1(\T) \text{ for some }{\kappa}>0.
\end{equation}
 We observe that, in particular, \eqref{m0mTintegrability} allows the initial density to vanish in a set of measure zero. In spite of this fact, our result also shows that $m$ becomes strictly positive instantly after the initial time. Moreover, in the case of \eqref{MainMFGSystem}, the density remains bounded below, and the solution remains smooth up to and including $t=T$. We refer to Section \ref{WeakSection} for the definition of a weak solution.
\begin{thm}\label{LocalTheorem}
Let $0<\alpha<1$, and assume that \hyperref[eq:mweak]{(W)},
\hyperref[eq:Hpp bd]{(H)} \hyperref[g increasing]{(G)}, 
and \hyperref[eq:ellipticity]{(E)} hold. Then the following statements hold: 
\begin{enumerate}[label=(\roman*)]
 \item There exists a weak solution 
 \[(u,m)\in (\text{{BV}}(Q_{T})\cap L^{\infty}(Q_T))\times(C([0,T],H^{-1}(\T))\cap L_{+}^{\infty}(Q_{T}))\]
to (\ref{PlanningSystem}). Moreover, $(u,m)\in C^{3,\alpha}_{\text{loc}}(Q_T)\times C^{2,\alpha}_{\text{loc}}(Q_T)$ and  $m>0$ in $(0,T)$. The function $m$ is unique, and $u$ is unique up to a constant.
   \item Assume, further, that the function $H$ satisfies, for each $(p,m)\in \mathbb{R} \times(0,\infty) $,
  \begin{equation} \label{RadialConditionH}H_p(p,m)p\geq 0.\end{equation}
   Then there exists a unique weak solution 
   \[(u,m)\in (\text{{BV}}(Q_{T})\cap L^{\infty}(Q_T))\times (C([0,T],H^{-1}(\T))\cap L_{+}^{\infty}(Q_{T}))\]
to (\ref{MainMFGSystem}). Moreover, $(u,m)\in C^{3,\alpha}_{\text{loc}}(\T \times (0,T])\times C^{2,\alpha}_{\text{loc}}(\T \times (0,T])$, and $m>0$ in $(0,T]$.\end{enumerate}
\end{thm}

Concerning the long time behavior of \eqref{MainTheorem}, it was shown by P. Cardaliaguet and P.J. Graber in \cite[Thm 5.1]{CardaliaguetGraber} that the rescaled solution $(x,s) \mapsto u(x,sT)/T$ converges, in a certain space $L^p(\T \times (\delta,1))$, to the map $\lambda (1-s)$, while the rescaled density $(x,s) \mapsto m(x,sT)$ converges in $L^p(\T \times (0,1))$ to the invariant measure $\mu \equiv 1$. Our third result shows that, when the marginals are strictly positive, a much stronger statement holds. That is, the solutions satisfy the turnpike property with an exponential rate of convergence, and the limit as $T\rightarrow \infty $ of the pair $(u(t)-\lambda (T-t),m(t))$ can be fully characterized as the solution to \eqref{InfiniteHorizonSystem}. We emphasize that this is a convergence result \emph{at the original time scale} (compare with \cite[Thm 2.6, Thm. 5.1]{CardPor} and \cite[Thm 4.1, Thm. 5.3]{CirantPor}).
\begin{thm} \label{thm: long time}
Assume that \hyperref[eq:M1]{(M)},
\hyperref[eq:Hpp bd]{(H)}, \hyperref[g increasing]{(G)},\hyperref[eq:ellipticity]{(E)}, and \hyperref[eq:H long time]{(L)}, hold, and let $T>1$.
Assume that $(u^T,m^T)$ is either the solution to \eqref{MainMFGSystem}, or the solution to \eqref{PlanningSystem} that satisfies $\int_{\T} v^T(\cdot,\frac{T}{2})=0$, where $$v^T(x,t):=u^T(x,t)-\lambda (T-t).$$ Then the following holds:
\begin{enumerate} [label=(\roman*)] \item There exist constants $C,\omega>0$, independent of $T$, such that
\[\|m^T(t)-1\|_{L^\infty(\T)} + \|u^T_x(t)\|_{L^\infty(\T)}\leq C(e^{-\omega t}+e^{-\omega(T-t)}), \,\,\,\,\, t\in [0,T].\]
Moreover, if $(u^T,m^T)$ solves $\eqref{MainMFGSystem}$, and $\eqref{RadialConditionH}$ holds, we have
\[\|m^T(t)-1\|_{L^\infty(\T)} + \|u^T_x(t)\|_{L^\infty(\T)}\leq Ce^{-\omega t}, \,\,\,\,\, t\in [0,T].\]
\item There exist functions $(v,\mu)$ such that, for each $T_0>0$,
\[v^T\rightarrow v \text{ in }   C^{3,\alpha}(\T\times [0,T_0])\text{ as } T\rightarrow \infty,  \]
and
\[m^T\rightarrow \mu  \text{ in }   C^{2,\alpha}(\T\times [0,T_0]) \text{ as } T\rightarrow \infty.\]
Moreover, one has \begin{equation} \label{eq:limit of v,mu} \lim_{t\rightarrow \infty} v(\cdot,t)=c,\,\,\,\lim_{t \rightarrow \infty}\mu(\cdot,t)=1 \text{ uniformly in } \T,\end{equation} 
where
\[c=\begin{cases} g(1) & \text{ if } (u^T,m^T) \text{ solves } \eqref{MainMFGSystem}, \\
0  &\text{ if } (u^T,m^T) \text{ solves }\eqref{PlanningSystem}.\end{cases} \]
Finally, $(v,\mu)$ is the unique classical solution to \eqref{InfiniteHorizonSystem} satisfying \eqref{eq:limit of v,mu} and
\begin{multline} \label{(v,mu) space}v\in W^{1,\infty}(\T \times (0,\infty)),\,\,\,\mu^{-1}\in L^{\infty}(\T \times(0,\infty)),\\ \mu-1 \in L^{1}(\T \times (0,\infty))\cap L^{\infty}(\T \times (0,\infty)) .\end{multline}
\end{enumerate}
\end{thm}
In particular, since the Hamiltonian $H(p,m)$ is non-separated, our results yield well-posedness and regularity of MFG systems with congestion, such as 
\begin{equation}
\begin{cases}
 -u_{t}+\frac{|u_x|^{2}}{2(m+c_{0})^{\alpha}}=f(m) \text{ in }  Q_T,\\ 
m_{t}-(\frac{m}{(m+c_{0})^{\alpha}}u_x)_x=0 \text{ in } Q_T, & 
\end{cases}\label{eq:mfg cong}
\end{equation}
where $0<\alpha<2$, $c_0\geq 0$, and $f'>0$.
Some of the key techniques used in \cite{Munoz1,Munoz2,Porretta2}, as well as in the present work, were developed by P.-L. Lions in his lectures at Collège de France \cite{Lions}, where he obtained several a priori estimates for the solutions to (\ref{PlanningSystem}), in the special case of a separated, quadratic Hamiltonian. The most important of these is the observation that the problems (\ref{MainMFGSystem}) and \eqref{PlanningSystem} can be transformed into a single quasilinear elliptic equation in $u$ after eliminating the variable $m$. Indeed, if one defines $H^{-1}$ by
\[
m=H^{-1}(p,H(p,m)),\]
then $m=H^{-1}(u_x,u_t)$ and the problem becomes
\begin{equation}
\tag{Q}\begin{cases}
Qu:=-\text{Tr}(A(Du)D^2u)=0 & \text{in }Q_{T},\\[1pt]
Nu:=B(x,t,u,Du)=0 & \text{on }\partial Q_{T},
\end{cases}\label{eq:quasilinear}
\end{equation}
where $Du=(u_x,u_t)$ and, for $(x,z,p,s)\in\T\times\mathbb{R}\times\mathbb{R}\times\mathbb{R},$
\begin{equation}
\tag{Q1}A(p,s)=\left(H_p+\frac{1}{2}mH_{mp},-1\right)\otimes\left(H_p+\frac{1}{2}mH_{mp},-1\right)-\begin{pmatrix}\frac{1}{4}m^2H_{mp}^2+mH_{m}H_{pp} & 0\\
0 & 0
\end{pmatrix},\label{eq:matrix-1}
\end{equation}
\begin{align}
\tag{B1}B(x,0,z,p,s)= & -s+H(p,m_{0}(x)),\;\label{eq:boundary1}
\end{align}
and
\begin{equation} \tag{B2} B(x,T,z,p,s)= \begin{cases}
s-H(p,g^{-1}(z)) & \text{in the case of (\ref{MainMFGSystem})} \\
s-H(p,m_T(x))& \text{in the case of (\ref{PlanningSystem})}.
\end{cases}\end{equation}
The condition for ellipticity, that is, for the matrix $A$ to be positive, is
\begin{equation}\label{UniquenessCondition}
-4mH_{m}H_{pp}> m^2H_{mp}^2,
\end{equation} 
which is also the well-known condition for uniqueness to (\ref{MainMFGSystem}) that follows from the Lasry-Lions monotonicity method (see, for instance, Lions and P.E. Souganidis \cite{Takis}). We remark from \eqref{UniquenessCondition} that, in particular, the strict positivity of the density is crucial for the regularizing properties of the system. The lower bounds on $m$ obtained in Corollary \ref{DisplacementConvexityCorBounds} and Proposition \ref{LocalLowerBoundDensityProp} both heavily rely on the one-dimensionality assumption, and this is the main obstacle to generalizing our results to higher dimensions. Indeed, in dimensions $d>1$, it remains an open question whether the existence of smooth solutions to local first order MFG systems can still be established if one removes or significantly weakens (\ref{BlowUpCondition}), or if $m_0$ is not assumed to be bounded away from $0$. Even for $d=1$, it is still unknown whether one can allow $m_0$ or $m_T$ to vanish in a set of positive measure.
 
 Section \ref{SectionAssumptions} contains all the assumptions that will be in place about the Hamiltonian $H$, as well as the initial and terminal data. In Section \ref{SectionDisplacement}, we establish an integral displacement convexity formula (see Proposition \ref{DisplacementProp}) that will allow us to bound the density $m$ in terms of its initial and terminal values. Section \ref{SectionAPrioriBounds} contains the necessary a priori estimates that are needed to prove the existence of classical solutions. In particular, we obtain, in Section \ref{SubsectionEpsilonPenalized}, estimates for an $\epsilon$--approximation of (\ref{PlanningSystem}) via standard MFG systems with a terminal condition of the type $u(\cdot,T)=g(\cdot,m(\cdot,T))$, which we require to prove existence for (\ref{PlanningSystem}). Finally, we provide a counterexample to existence of solutions to (\ref{MainMFGSystem}) when the terminal cost function $g$ is also allowed to depend on the space variable (see Proposition \ref{PropositionCounterexample}). In Sections \ref{SectionExistence} \ref{WeakSection}, and \ref{LongSection}, we prove our main results, Theorems \ref{MainTheorem}, \ref{LocalTheorem}, and \ref{thm: long time}, respectively.

We remark that, in the special case of a separated Hamiltonian, the displacement convexity estimates of Section \ref{SectionDisplacement}  were first obtained by D. A. Gomes and T. Seneci in \cite{Gomes}. Further estimates on the density using displacement convexity were also obtained by T. Bakaryan, R. Ferreira, and Gomes in \cite{Gomes2}, and by Porretta in \cite{Porretta2} (see also Lavenant, Santambrogio \cite{SantambrogioLavenant}). Weak solutions, as defined in Section \ref{WeakSection}, have been widely studied for both \eqref{MainMFGSystem} (see \cite{Cardaliaguet, CardaliaguetGraber, CardaliaguetGraberPorrettaTonon, CardaliaguetPorretta, Munoz1}) and \eqref{PlanningSystem} (see \cite{weakMFGP1,Porretta2, weakMFGP2}). For classical solutions in the time-independent case we refer to Evans \cite{Evans} and Gomes, Mitake \cite{Gomes3}. Concerning the study of the long time behavior of solutions, specifically the second part of Theorem \ref{thm: long time}, we follow the program developed by Porretta and Cirant in \cite{CirantPor}, where a similar analysis was performed for second-order MFG systems, and, unlike the earlier work \cite{CardPor}, does not involve the use of the master equation (see also \cite{CardLasLiPor,PorrettaLT}).

\subsection*{Notation} \label{SubsectionNotation}
Let $d,k\in \mathbb{N}$. For $T>0$, we denote by $Q_T:=\T\times (0,T)$, $\overline{Q_T}:=\T\times [0,T]$ and $\partial Q_T:=\T\times \{0,T\}$. For $\alpha\in (0,1],\, T>0$, and $\Omega\subset \R^d$ we denote by  $C^{k+a}(\Omega),$ the standard space of $k$ times differentiable scalar functions with $\alpha-$H\"older continuous $k^{\text{th}}$ order derivatives, with the usual norm. Furthermore, we denote by $C_{loc}^{k+\alpha}(\Omega)$ the functions $u$ that belong to $C^{k+\alpha}(K)$, for all compact sets $K\subset \Omega$. For functions $u:   \T\times [0,T]\rightarrow \R$, we denote by $\text{osc } u:=\max\limits_{(x,t)\in \T\times [0,T]}u(x,t)-\min\limits_{(x,t)\in \T\times [0,T]}u(x,t)$, $Du(x,t):=(u_x(x,t),u_t(x,t))$. We denote by $H^{-1}(\T)$ the dual space of
the Sobolev space $H^{1}(\T)$, and the space of $H^{-1}(\mathbb{T}^{d})$--valued
$\alpha$--Hölder continuous functions by $C^{0,\alpha}([0,T];H^{-1}(\mathbb{T}^{d}))$. We write $C=C(K_{1},K_{2},\ldots,K_{M})$ for a positive constant $C$ depending monotonically on the non-negative quantities $K_{1},\ldots,K_{M}.$ BV$(Q_{T})$ denotes the space
of functions of bounded variation, and $L_{+}^{\infty}(Q_{T}$) consists
of the functions $m\in L^{\infty}(Q_{T}$) such that $m\geq0$ a.e. in $Q_{T}$.

\section{Assumptions} \label{SectionAssumptions}
In what follows, $C_0$ and $\gamma$, $\alpha$ are positive constants, with $\gamma>1$, and $0<\alpha<1$. Moreover, $\overline{C}:(0,\infty) \rightarrow (0,\infty)$ is a continuous, strictly positive function. Except when explicitly stated, assumptions \hyperref[eq:M1]{(M)}, \hyperref[eq:Hpp bd]{(H)}, \hyperref[g increasing]{(G)}, and \hyperref[eq:ellipticity]{(E)} will be in place throughout the paper. 

\begin{itemize}
\item[(M)] (Assumptions on $m_{0}$ and $m_T$ for classical solutions) The given functions $m_{0}$ and $m_T$ satisfy
\begin{equation}
\tag{M1} m_{0},m_T\in C^{2,\alpha}(\mathbb{T}),\;m_{0},m_T>0,\text{ and }\int_{\mathbb{T}}m_{0}=\int_{\mathbb{T}}m_T=1.\label{eq:M1}
\end{equation}

\item[(H)]  (Assumptions on $H$) The functions $H$, $H_p$, and $H_{pp}$
are in $C^4(\mathbb{R} \times (0,\infty))$, and $H_{m}<0$.
Moreover, for $(p,m)\in \mathbb{R}\times(0,\infty)$,
\begin{equation}
\tag{H1}\frac{1}{C_0}(1+|p|)^{\gamma-2}\leq H_{pp}\leq \overline{C}(m)(1+|p|)^{\gamma-2},\label{eq:Hpp bd}
\end{equation}
\begin{equation}
\tag{H2}pH_p\geq (1+\frac{1}{C_0})H -\overline{C}(m),\label{eq:H superlinear growth}    \end{equation}
\begin{equation}
\tag{H3}|H_{ppp}|\leq \overline{C}(m)(1+|p|)^{\gamma-3},\label{eq:Hppp bd}    \end{equation}
\begin{equation}
\tag{HM1} |H_{m}|\leq \overline{C}(m)(1+|p|)^\gamma,\label{eq:Hm growth}
\end{equation}

\begin{equation}
\tag{HM2}m|H_{mm}|\leq-\overline{C}(m)H_{m},\;\;\;|p|\|H_{mp}|\leq -\overline{C}(m)H_{m},\;\;\;m|p||H_{mmp}|\leq -\overline{C}(m) H_{m},\label{eq:Hm bd}
\end{equation}
\begin{equation}
\tag{HM3} |H_{mpp}|\leq \overline{C}(m)(1+|p|)^{\gamma-2} \label{eq:Hmpp bd}
\end{equation}
\item[(G)] (Assumptions on $g$) The function $g:(0,\infty)\rightarrow\mathbb{R}$
is four times continuously differentiable and satisfies, for all $m>0$, \begin{equation} \tag{G1} \label{g increasing}
    g'(m)>0.
\end{equation}

\item[(E)] (Ellipticity of the system) The function $H$ satisfies, for $m>0$, the condition
\begin{equation}
\tag{E1}-4mH_{m}H_{pp}\geq\left(1+\frac{1}{C_{0}}\right)m^2H_{mp}^2.\label{eq:ellipticity}
\end{equation}
\item[(W)](Assumptions on $m_0$, $m_T$, $H$, and $g$ for weak solutions) The functions $m_0$ and $m_T$ satisfy, for some $\kappa>0$, 
\begin{equation} \tag{MW} \label{eq:mweak} m_0,\,m_T\in L^{\infty}(\T),\,\, m_0,\,m_T\geq 0,\,\, \intT m_0=\intT m_T=1, \text{ and }\frac{1}{m_0^{\kappa}},\frac{1}{m_T^{\kappa}}\in L^1(\T),\end{equation}
$H$ satisfies, for some constant $s\in (-\kappa-1,\kappa-1)$, and for $(p,m)\in \mathbb{R}\times (0,\frac{1}{C_0})$,
\begin{equation} \label{eq:Hm growth weak}\tag{HW}-H_m(0,m)\leq C_0 m^{s},\, -H_m(p,m)\geq \frac{1}{C_0}m^{s}, \end{equation}
and $g$ satisfies
\begin{equation} \label{eq:g(0)condition}\tag{GW}\lim_{m\rightarrow 0^+}g(m)>-\infty. \end{equation}
\item[(L)](Assumption on $H$ for the long time average) The function $H$ satisfies, for $(p,m)\in \mathbb{R}\times(0,\infty) $,
\begin{equation} \label{eq:H long time}\tag{HL} -mH_m(p,m)\geq \frac{1}{\overline{C}(m)}.\end{equation}
\end{itemize}

\section{Displacement convexity and estimates on the density} \label{SectionDisplacement}
To obtain estimates for the density at interior times, we will prove an integral formula which, in particular, implies that the quantity
\[\int_{\T}h(m(x,\cdot))dx\]
is a convex function in $[0,T]$ whenever $h$ is convex, provided that (\ref{UniquenessCondition}) holds.
\begin{prop}\label{DisplacementProp}

Let $(u,m)\in  C^{2}(\overline{Q}_T)\times C^{1}(\overline{Q}_T)$ be a classical solution to
\begin{equation} \label{MFGSystemNoTerminalCondition}
    \begin{cases}-u_t+H(u_x,m)=0, & \text{ in }Q_{T}\\
        m_t-(mH_p(u_x,m))_x=0, & \text{ in }Q_{T}\\
        m(\cdot,0)=m_0, & \text{ in }\T,   \end{cases}
\end{equation}
and let $h \in W^{2,\infty}(\R)$. Then 

\begin{multline} \label{DisplacementFormula}
   \frac{d^2}{dt^2} \int_{\T} h(m(x,t))dx=\intT h''(m)\Big(m_t-m_x(H_p+\frac{m}{2}H_{pm})\Big)^2dx\\ -\intT h''(m)(m_x)^2\Big(\frac{m^2}{4}H_{pm}^2+mH_{pp}H_m \Big)dx.
\end{multline}
Moreover, there exists $C=C(C_0)$ such that, if $h''>0$,
\[ \label{DisplacementInequality}
   \frac{d^2}{dt^2} \int_{\T} h(m(x,t))dx\geq \frac{1}{C}\intT h''(m) (- mH_mH_{pp}m_x^2+m^2H_{pp}^2u_{xx}^2)dx.
\]
\end{prop} 
\begin{proof}
Let $g:\R\rightarrow \R$, be a smooth function. Since $m$ satisfies the continuity equation, the following holds for each $t\in [0,T]$:
\begin{equation} \label{displacement1}
    \intT \Big(m_t(x,t)-(m(x,t)H_p(u_x,m(x,t)))_x\Big)\Big(\pt g(m(x,t)) -(g(m(x,t))H_p(u_x,m(x,t)))_x\Big)dx=0 .
\end{equation}
Expanding equation (\ref{displacement1}), we obtain 
\begin{multline*}
    0=\intT (m_t-m_x(H_p+mH_{pm})-mH_{pp}u_{xx})(g'(m)m_t -m_x(g'(m)H_p+g(m)H_{pm})-g(m)H_{pp}u_{xx})dx\\
    =\intT g'(m)(m_t)^2-m_tm_x\Big[2g'(m)H_p+\Big(g'(m)m+g(m)\Big)H_{pm}\Big]\\
    +m_xH_{pp}u_{xx}\Big[H_p\Big(g'(m)m+g(m)\Big)+2g(m)mH_{pm}\Big]\\
    +m_x^2\Big[\Big(H_p+mH_{pm}\Big)\Big(g'(m)H_p+g(m)H_{pm}\Big)\Big]\\
    -m_tH_{pp}u_{xx}\Big[H_p\Big(g'(m)m+g(m)\Big)+2g(m)mH_{pm}\Big]\\
    +g(m)m\Big(H_{pp}u_{xx}\Big)^2dx
    =A_1-A_2+A_3+A_4-A_5+A_6.
\end{multline*}
We split term $A_3$ as follows
\[A_3=\intT m_xH_{pp}H_pu_{xx}\Big( g'(m)m+g(m)\Big) dx+2\intT g(m)m_x
mH_{pm}
H_{pp}u_{xx}dx=A_{3.1}+A_{3.2}.\]
From the continuity equation, we have that 
\[mH_{pp}u_{xx}=m_t -m_x(H_p+mH_{pm}).\]
Hence, terms $A_{3.2}$ and $A_6$ can be written as
\[A_{3.2}=2\intT m_tm_x
H_{pm}
g(m) dx-2\intT (m_x)^2
H_{pm}
\Big(g(m)H_p+mg(m)H_{pm}\Big)dx=A_{3.2.1}-A_{3.2.2}\]
\[A_6=\intT \frac{g(m)}{m}\Big[m_t-m_x\Big(H_p+mH_{pm}\Big)\Big]^2dx\]
\[=\intT\frac{g(m)}{m}(m_t)^2-2\frac{g(m)}{m}m_t m_x\Big(H_p+mH_{pm}\Big)+\frac{g(m)}{m}(m_x)^2\Big(H_p+mH_{pm}\Big)^2dx=A_{6.1}-A_{6.2}+A_{6.3}.\]
From the Hamilton-Jacobi (HJ for short) equation, we have that 
\[H_{p}u_{xx}=u_{xt}-H_mm_x.\]
Therefore, $A_{3.1}$ may be written as
\[A_{3.1}=\intT m_xH_{pp}u_{xt}\Big(g'(m)m+g(m) \Big) dx-\intT (m_x)^2H_{pp}H_m\Big(g'(m)m+g(m) \Big) dx=A_{3.1.1}-A_{3.1.2}\]
We now begin by grouping together terms $A_5$, and $A_{3.1.1}$, which yields, for $L(m)=g(m)m,\,\, L'(m)=g(m)+mg'(m)$,
\begin{align*}
    -A_5+A_{3.1.1}=&\intT m_x\Big(g(m)+mg'(m)\Big)H_{pp}u_{xt} -\Big(g(m)+mg'(m)\Big)m_t H_{pp}u_{xx}dx\\
    =&\intT -\pt(L(m))(H_p)_x+L'(m)m_t H_{pm}m_x+(L(m))_x\pt(H_p)-L'(m)m_xm_t H_{pm}dx \\
    =&\intT \pt((L(m))_x)H_p+(L(m))\pt (H_p)dx=\frac{d}{dt} \intT (L(m))_xH_pdx,
\end{align*}
Next, we group together all the terms with $m_t m_x$ factor, namely $A_{2}$, $A_{3.2.1}$, and $A_{6.2}$, which yields
\[-A_2+A_{3.2.1}-A_{6.2}=-\intT 2m_t m_x \left(g'(m)+\frac{g(m)}{m}\right)\left(H_p+\frac{m}{2}H_{pm}\right)dx.\]
Collecting the terms involving $(m_t)^2$, namely terms $A_{1}$ and $A_{6.1}$, we obtain
\[A_1+A_{6.1}=\intT (m_t)^2\left(g'(m)+\frac{g(m)}{m}\right)dx.\]
Finally, we group together the terms involving $m_x^2$, namely $A_{4}$, $A_{3.2.2}$, $A_{6.3}$, and $A_{3.1.2}$: 
\begin{multline*}
    A_4-A_{3.2.2}+A_{6.3}-A_{3.1.2}=\\
    \intT (m_x)^2\Big[\Big(g'(m)+\frac{g(m)}{m}\Big)\Big(H_p+\frac{m}{2}H_{pm}\Big)^2\Big]dx\\
    -\intT (m_x)^2\Big[\Big(g'(m)+\frac{g(m)}{m}\Big)\Big(\frac{m^2}{4}H_{pm}^2+mH_{pp}H_m\Big)\Big]dx.
\end{multline*}
Thus, putting everything together, we obtain 
\begin{multline}-\frac{d}{dt} \intT (L(m))_xH_p dx=\intT \Big(g'(m)+\frac{g(m)}{m}\Big)\left(m_t -m_x\Big(H_p+\frac{m}{2}H_{pm}\Big)\right)^2dx \\
-\intT m_x^2\left(g'(m)+\frac{g(m)}{m}\right)\Big(\frac{m^2}{4}H_{pm}^2+mH_{pp}H_m\Big)dx.\end{multline}
Next, notice that for a smooth function $h:\R\rightarrow \R$, we have
\[\frac{d}{dt} \intT h(m)dx=\intT (h(m))_xH_p+mh'(m)(H_p)_xdx=\intT (h(m)-h'(m)m)_x H_{p}dx.\]
Thus, if we require that 
\[-L(m)=h(m)-h'(m)m,\]
we obtain 
\[-\frac{d}{dt}\intT (L(m))_x H_pdx=\frac{d^2}{dt^2}\intT h(m)dx.\]
The relation between $h,g$ is
\[mg(m)=h'(m)m-h(m),\]
therefore
\[g(m)=-\frac{h(m)}{m}+h'(m),\]
and, thus,
\[g'(m)+\frac{g(m)}{m}=-\frac{h'(m)}{m}+\frac{h(m)}{m^2}+h''(m)-\frac{h(m)}{m^2}+\frac{h'(m)}{m}=h''(m),\]
from which \eqref{DisplacementFormula} follows.\\
Now, setting $r=1-\frac{1}{1+C_0^{-1}},$ we have
\[-\frac{m^2}{2}H_{pm}^2-mH_mH_{pp}=-\frac{m^2}{2}H_{pm}^2-(1-r)mH_mH_{pp}-rmH_mH_{pp},\]
and so, applying \hyperref[eq:ellipticity]{(E)}, and multiplying by $h''(m)m_x^2$, \eqref{DisplacementFormula} yields
\begin{equation}\label{eq: displ 1} \frac{d^2}{dt^2} \int_{\T} h(m(x,t))dx\geq \int_{\T}-rh''(m)mH_mH_{pp}m_x^2.\end{equation}
On the other hand, we infer from \hyperref[eq:ellipticity]{(E)} that
\begin{multline}\Big(m_t-m_x(H_p+\frac{m}{2}H_{pm})\Big)^2-m_x^2\left(\frac{m^2}{2}H_{pm}^2+mH_mH_{pp}\right) \\ \geq \Big(m_t-m_xH_p-\frac{m_xm}{2}H_{pm}\Big)^2 + \frac{1}{C_0}\left(\frac{m_xm}{2}H_{pm}\right)^2
=(m_t-m_xH_p)^2-2(m_t-m_xH_p)\frac{m_xm}{2}H_{pm}\\+(1-r)^{-1}\left(\frac{m_xm}{2}H_{pm}\right)^2=r(m_t-m_xH_p)^2+\left((1-r)^{\frac{1}{2}}(m_t-m_xH_p)-(1-r)^{-\frac{1}{2}}\frac{m_xm}{2}H_{pm}\right)^2\\
\geq r(m_t-m_xH_p)^2=rm^2H_{pp}^2u_{xx}^2.\end{multline}
where the last equality follows from the equation of $m.$ As before, multiplying by $h''(m)$ then yields 
\begin{equation} \label{eq: displ 2}\frac{d^2}{dt^2} \int_{\T} h(m(x,t))dx\geq \int_{\T} rh''(m)m^2H_{pp}^2u_{xx}^2. \end{equation}
Combining \eqref{eq: displ 1} and \eqref{eq: displ 2}, we conclude that \eqref{DisplacementInequality} holds.
\end{proof}
It now follows readily that the density of the solution is bounded above and below in terms of the initial and terminal densities. 
\begin{cor}\label{DisplacementConvexityCorBounds}
Let $(u,m)\in C^{2}(\overline{Q}_T)\times C^{1}(\overline{Q}_T)$ be a classical solution to (\ref{MainMFGSystem}) or (\ref{PlanningSystem}). Then, if $c_1:=\min(\min m_0,\min m(\cdot,T))$, $C_1=\max(\max{m_0},\max{m(\cdot,T)})$, we obtain that 
\begin{equation} \label{eq:APriorim} c_1\leq m(x,t)\leq C_1, \text{ for all }(x,t)\in \overline{Q}_T.\end{equation}
\end{cor} 
\begin{proof}
The proof follows directly from Proposition \ref{DisplacementProp} above. Indeed, note that, in view of \hyperref[eq:ellipticity]{(E)}, for any convex function $h$, the map
\[C(t):=\intT h(m(x,t))dx\]
is convex, and thus
\[C(t)\leq \max (C(0),C(T)), \text{ for all }t\in [0,T].\]
Hence, setting $h_p(m)=m^p$ and letting $p\rightarrow -\infty$ yields the result for the lower bound, whereas letting $p\rightarrow +\infty$ yields the upper bound.
\end{proof}
\begin{rmk}
For dimensions $d>1$, formula (\ref{DisplacementFormula}) is no longer true. If one repeats the same argument, the issue will arise at the term $A_{6.2}$. However, in the case of a separated Hamiltonian, i.e. $H(p,m)\equiv H(p)-f(m)$, one still obtains the weaker formula
\begin{multline}\frac{d^2}{dt^2}\int_{\T} h(m(x,t))dx=\int_{\T} ((h''(m)m^2-h'(m)m+h(m))(\text{tr}(D^2_{pp}HD_{xx}^2u))^2\\
+(h'(m)m-h(m))\text{tr}((D^2_{pp}HD_{xx}^2u)^2)+h''(m)mf'(m)|Dm|^2)dx.\end{multline}
In this higher-dimensional setting, it is no longer true that the left hand side is convex whenever $h$ is convex. In particular, the statement is false for negative powers of $m$, but true for positive powers. Thus, from the proof of Corollary \ref{DisplacementConvexityCorBounds} we see that the upper bound on $m$ still holds (see \cite{Gomes}). 
\end{rmk}

\section{Estimates on the solution and the terminal density} \label{SectionAPrioriBounds}
In this section we obtain the necessary a priori $L^{\infty}-$bounds on $u$, $Du$, and $m(\cdot,T)$ for solutions to both (\ref{MainMFGSystem}) and (\ref{PlanningSystem}). Combined with the results of the previous section, this will yield global upper and lower bounds on the density. In order to treat the setting of Theorem \ref{LocalTheorem}, where the density may vanish at $\{0,T\}$, we also obtain $L^\infty$-bounds on $u$ that do not depend on the quantities $(\min m_0)^{-1}$, $(\min m_T)^{-1}$.
\begin{prop} \label{mAnduBounds}
Let $(u,m)\in C^{2}(\overline{Q}_T)\times C^{1}(\overline{Q}_T)$ be a classical solution to (\ref{MainMFGSystem}), and let $c_1=\min m_0, C_1=\max{m_0}$. Then, for each $(x,t)\in \overline{Q}_T,$
\begin{equation} \label{mTbounds} c_1\leq m(x,T)\leq C_1 \text{ for all }x\in \T,\end{equation}
\begin{equation} \label{uboundsMFG1}H(0,c_1)(t-T)+g(c_1)\leq u(x,t)\leq H(0,C_1)(t-T)+g(C_1) \text{ for all }(x,t)\in \overline{Q}_T,\end{equation}
and
\begin{equation} \label{uboundsMFG2}-\int_{t}^T H(0,\min_{\T}(m(\cdot,s))ds+g(c_1)\leq u(x,t) \leq-\int_{t}^T H(0,\max_{\T}(m(\cdot,s))ds+g(C_1).\end{equation}
\end{prop}
\begin{proof}
We will only show the lower bounds, since the argument for the upper bounds is completely symmetrical. We fix $\delta >0$ and let $\epsilon>0$ be such that 
\[H(0,c_1)-H(0,c_1-\delta)<-\epsilon T, \text{ for all }x\in \T.\]
We define
\[w^{\epsilon,\delta}(t):={H(0,c_1-\delta)}(t-T)+\frac{\epsilon}{2}(t-T)^2+g(c_1-\delta),\]
and note that
\[w_{xx}=0, w_{x,t}=0, w_{tt}=\epsilon.\]
The function $v^{\epsilon,\delta}(x,t):=u(x,t)-w^{\epsilon,\delta}(t)$ has a minimum at some $(x_0,t_0)\in \overline{Q}_T$. If we first assume that $t_0\in (0,T)$, then it follows that
\[ D^2u-D^2w^{\epsilon,\delta}\geq 0,\]
which, in view of \eqref{eq:quasilinear}, implies
\[0=-\text{Tr}(AD^2u)\leq -\text{Tr}(AD^2w^{\epsilon,\delta})=-\epsilon<0,\]
a contradiction. On the other hand, assume that $t_0=0$. Then,
\[u_t(x_0,0)\geq  w^{\epsilon,\delta}_t(x_0,0),\, \, u_x(x_0,0)= w_x^{\epsilon,\delta}(0)=0,\]
and thus
\[0=-u_t(x_0,t_0)+H(0,m_0(x_0))\leq - w^{\epsilon,\delta}_t(0)+H(0,m_0(x_0))=-H(0,c_0-\delta)+H(0,m_0(x_0))+\epsilon T\]
\[\leq -H(0,c_1-\delta)+H(0,c_1)+\epsilon T<0,\]
by our choice of $\epsilon$, which is a contradiction. Hence, the minimum must be achieved at $t_0=T$. At that point, we have
\[u_t(x_0,T)\leq  w^{\epsilon,\delta}_t(T),\,\,  u_x(x_0,T)=w_x^{\epsilon,\delta}(T)=0.\]
Consequently, 
\[u(x_0,T)=g(H^{-1}(0,u_t(x_0,T)))\geq g(H^{-1}(0, w^{\epsilon,\delta}_t(T)))=g(H^{-1}(0,H(0,c_1-\delta)))\]
\[=g(c_1-\delta)=w^{\epsilon,\delta}(T).\]
We have thus shown that 
\[u(x,t)\geq w^{\epsilon,\delta}(t) ,\text{ for all }(x,t)\in \overline{Q}_T.\]
Letting $\epsilon\rightarrow 0$, and then $\delta\rightarrow 0$, yields the lower bound in \eqref{uboundsMFG1}
In particular, for $t=T$, we have
\[g(m(x,T))\geq g(c_1) \text{ for all } x \text{ in }\T,\]
which proves the lower bound in \eqref{mTbounds}. Now, we define 
\[w(t)=-\int_t^T H(0,c(s))ds+g(c_1),\]
where $c(s):=\min\limits_{\T}\{m(\cdot,s)\}$ is the running minimum of the density. We observe that the function $v(x,t)=u(x,t)-w(t)$ satisfies $v_t=u_t-H(0,c(t))$, $v_x=u_x$. Thus, for any $\epsilon>0$, at any extremum point of $v-\epsilon t$, one has $v_t=H(0,m)-H(0,c(t))-\epsilon<0$. Letting $\epsilon \rightarrow 0$ thus implies that $v$ achieves its minimum at $t=T$. Therefore, using \eqref{mTbounds}, we obtain
\[u(x,t)-w(t)\geq \min_{\T}g(m(\cdot,T))-g(c_1)\geq 0,\]
and this is precisely the lower bound in \eqref{uboundsMFG2}.
\end{proof}
Now, for solutions to (\ref{PlanningSystem}), we do not need to estimate the terminal density, as it is part of the given data. Concerning $u$, since the solution is only unique up to a constant, we may only bound the oscillation of $u$, and this is done in the following proposition.

\begin{prop}\label{OscillationProp} Let $(u,m)\in C^{2}(\overline{Q}_T)\times C^{1}(\overline{Q}_T)$ solve (\ref{MFGSystemNoTerminalCondition}). There exists a constant $C>0$, with \[C=C\left(C_0, \int_{0}^T |H(0,\min_{\T}m(\cdot,s))|ds, \overline{C}(\max\limits_{\overline{Q}_T} m)\right),\]
such that
\[\text{osc}_{\overline{Q}_T} u\leq C(T+T^{-\frac{1}{\gamma-1}}+\int_0^T|H(0,\min_{\T}m(\cdot,s)|ds).\]
\end{prop}
\begin{proof}
We define the functions $c$ and $w$, for $t\in[0,T]$, by
\[c(t)=\min_{\T}m(\cdot,t),\,\,\,w(t)=-\int_t^T H(0,c(s))ds.\]
Arguing as in the proof of \eqref{uboundsMFG2}, we obtain
\begin{equation} \label{1}\max_{\overline{Q}_T}(u-w)=\max_{\T}\left( (u(\cdot,0)-w(0)\right),\,\min_{\overline{Q}_T}(u-w)=\min_{\T}\left(u(\cdot,T)-w(T)\right).\end{equation}
Now, in view of (\ref{eq:Hpp bd}) and Proposition \ref{mAnduBounds},
$0=-u_t+H(u_x,m)\geq -u_t+\frac{1}{C}|u_x|^{\gamma}-C.$
Next, we define $\gamma'$ by $\frac{1}{\gamma}+\frac{1}{\gamma'}=1$. By the Hopf-Lax formula, the function
\[v(x,t)=\min_{y\in\mathbb{R}}\Big(\left(\frac{C}{\gamma}\right)^{\frac{\gamma'}{\gamma}}(T-t)\frac{|x-y|^{\gamma'}}{\gamma'(T-t)^{\gamma'}}+C(T-t)+u(y,T)\Big)\]
then solves, in $\overline{Q}_T$,
 \[-v_t(x,t)+\frac{1}{C}|v_x|^{\gamma}-C=0,\, \,
         v(\cdot,T)=u(\cdot,T),
\]
and, thus, by the comparison principle,
    \[u\leq v.\]
On the other hand, up to increasing the constant $C$,
\[
    v(x,0)\leq \frac{C}{ T^{\gamma'-1}}+CT+\min_{\T} u(\cdot, T),
\]
and so
\[\max_{\T} u(\cdot, 0)\leq \max_{\T}v(\cdot,0)\leq \frac{C}{ T^{\gamma'-1}}+CT+\min_{\T} u(\cdot, T).\]
In view of (\ref{1}), we obtain
\[\text{osc}_{\overline{Q}_T} (u-w)\leq \frac{C}{ T^{\gamma'-1}}+CT+w(T)-w(0), \]
and, thus,
\[\text{osc}_{\overline{Q}_T} u \leq \frac{C}{ T^{\gamma'-1}}+CT+2\cdot\text{osc}_{\overline{Q}_T}w \leq \frac{C}{ T^{\gamma'-1}} +CT+2 \int_0^T |H(0,c(s))|ds. \]\end{proof}

We finally obtain a priori estimates on the gradient of $u$, while simultaneously treating the case of (\ref{MainMFGSystem}) and (\ref{PlanningSystem}). The proof closely follows \cite[Lem. 3.8]{Munoz1} and \cite[Lem 3.3]{Munoz2}, but yields a slightly stronger estimate due to the $d=1$ assumption (see (\ref{uxx bd})). For the purpose of studying the long time behavior, we will keep track of the dependence of $T$ for large values of $T$.
\begin{prop}\label{GradientEstimateProp} Let $(u,m)\in C^{3}(\overline{Q}_T)\times C^{2}(\overline{Q}_T)$ be a classical solution to (\ref{MainMFGSystem}) or (\ref{PlanningSystem}). There exists a constant $C>0$, with
\begin{multline*}
    C=C\Big(C_0,T,T^{-1},\text{osc }u,\gamma,\|m\|_{L^{\infty}(\overline{Q}_T)},\|m^{-1}\|_{L^{\infty}(\overline{Q}_T)},\\
    \|(m_0)_x\|_{L^\infty(\T)},\|(m_T)_x\|_{L^\infty(\T)},\|\overline{C}\|_{L^\infty[\min m,\max m]}\Big)
\end{multline*}
such that
\[\|Du\|_{L^{\infty}(\overline{Q}_T)}^{\gamma}\leq C.\]
\end{prop}
\begin{proof}
Since $u_t=H(u_x,m)$, and $m$ is bounded above and below, we infer from \eqref{eq:Hpp bd} and \eqref{eq:H superlinear growth} that it is enough to show that \[||u_x||_{L^{\infty}(\overline{Q}_T)}\leq CT^2.\]
We let
\[
\tilde{u}=u-\min u +1-\frac{(\text{osc }u+2)}{T}(T-t),
\]
and note that the function $\tilde{u}$ has been constructed to satisfy
\[
|\tilde{u}|\leq 1+\text{osc }u,\quad\tilde{u}(\cdot,0)\leq-1,\;\tilde{u}(\cdot,T)\geq1.
\]
Define
\[v(x,t)=\frac{1}{2}u_x^2+\frac{k}{2} \tilde{u}^2,\]
where $k=\|u_x\|^\frac{3}{2}_{\overline{Q}_T}$. Let $(x_0,t_0)\in \overline{Q}_T$ be a point where $v$ achieves its maximum value. With no loss of generality, we may assume that $p=u_x(x_0,t_0)$ satisfies \[|p|\geq1,\; |p|^2\geq \frac{1}{2}\|u_x\|^2.\]\\
We remark here that throughout the proof, the constant $C$ is subject to increase from line to line. \\
\textbf{Case 1:\;}$t_0=T$. For this case we consider the linearization of the HJ equation,
\[T_{u}v=-v_{t}+H_p(u_x,m)v_x.\]
Since $v_x=0$ and $v_t\geq0,$
\begin{multline} \label{grad1}
0\geq T_{u}v=  T_{u}\left(\frac{1}{2}|u_x|^{2}\right)+k\tilde{u}(-\tilde{u}_{t}+H_p u_x)\\
=  -H_{m}u_xm_x+k\tilde{u}(-u_{t}+H_p p-C) \geq -H_m u_x m_x +k\tilde{u}(\frac{1}{C_0}H)-C k \tilde {u} \\
\geq -H_m u_x m_x +k\tilde{u}\frac{1}{C_0}\left(\frac{1}{\overline{C}(m)}|p|^{\gamma}-\overline{C}(m)\right)-C |p|^{\frac{3}{2}}
\geq  -H_m u_x m_x +\frac{1}{C}|p|^{\gamma+\frac{3}{2}}-  C|p|^{\frac{3}{2}}.
\end{multline}
If $(u,m)$ solves (\ref{MainMFGSystem}), then 
\[-H_m u_x m_x=-\frac{H_m}{g'}|p|^2>0.\] On the other hand, if $(u,m)$ solves (\ref{PlanningSystem}), then \begin{equation} \label{grad2} |-H_m u_x m_x|\leq C\|(m_T)_x\|_{\infty}|p|^{\gamma+1}.\end{equation}
In either case, (\ref{grad1}) then implies 
\[|p|\leq C.\]
\textbf{Case 2:\;}$t_0=0$. Regardless of whether $(u,m)$ solves (\ref{MainMFGSystem}) or (\ref{PlanningSystem}), this case is dealt with in the same way as was done for $t_0=T$ when $(u,m)$ solved (\ref{PlanningSystem}), because, in view of \ref{eq:Hm bd}, we then have the bound
\[|-H_m u_x m_x|\leq C\|(m_0)_x\|_{\infty}|p|^{\gamma+1}.\]
\textbf{Case 3: \;}$0<t_0<T$. We first observe that, since $v_x=0$, we have
\[u_x u_{xx}=-k\tilde{u}u_x,\]
and, thus,
\begin{equation} \label{uxx bd}|u_{xx}|\leq C k. \end{equation}
We consider the linearization of (\ref{eq:quasilinear}), namely
\[L_u(w)=-\text{Tr}(A(Du)D^2w)-D_q\text{Tr}(A(Du)D^2u)\cdot Dw.\]
Through direct computation, using (\ref{eq:matrix-1}), one obtains
\begin{equation} \label{gradientproof1}
    L_u\left(\frac{1}{2}u_x^2\right)=-\left|-u_{xt}+\left(H_p+\frac{1}{2}mH_{mp}\right)u_{xx}\right|^2+\frac{1}{4} m^2H_{mp}^2u_{xx}^2-mH_mH_{pp} u_{xx}^2,
\end{equation}
and
\begin{equation} \label{gradientproof2}  L_u \left( k \frac{1}{2}\tilde{u}^2 \right)=-k\left|-\tilde{u}_{t}+(H_p+\frac{1}{2}mH_{mp})u_{x}\right|^2+k\frac{1}{4} m^2H_{mp}^2u_{x}^2-kmH_mH_{pp} u_{x}^2+E_1+E_2+E_3+E_4,\end{equation}
where
\[E_1=2\left(-u_{xt}+\left(H_p+\frac{1}{2}mH_{mp}\right)u_{xx}\right)\left(H_{pp}+\frac{1}{2}mH_{mpp}\right)k\tilde{u}u_x, \]
\[E_2=\left( \frac{1}{2}H_{mp}H_{mpp}+mH_{mp}H_{pp}+mH_mH_{ppp}\right)u_{xx}k\tilde{u}u_x,\]
\[E_3=\left(-u_{xt}+\left(H_p+\frac{1}{2}mH_{mp}\right)u_{xx}\right)\frac{2}{H_m}\left(H_{pm}+\frac{1}{2}\left(mH_{mmp}+H_{mp}\right)\right)k\tilde{u}(-\tilde{u}_t+H_p u_x)\]
\begin{multline*}
    E_4=\\
    \frac{1}{H_m}\left(\frac{1}{2}(mH_{mp}^2+m^2H_{mp}H_{mmp})+mH_{mm}H_{pp}+mH_mH_{mpp}+H_mH_{pp} \right)u_{xx}k\tilde{u}(-\tilde{u}_t+H_p u_x).
\end{multline*}
Now we estimate each of the $E_i$. By Young's inequality, we obtain
\[|E_1|\leq \frac{1}{4}\left|-u_{xt}+\left(H_p+\frac{1}{2}mH_{mp}\right)u_{xx}\right|^2  +C|H_{pp}+\frac{1}{2}mH_{mpp}|^2 k^2u_x^2\tilde{u}^2.
\]
As a result of (\ref{eq:Hpp bd}), and (\ref{eq:Hmpp bd}), we thus obtain
\begin{equation} \label{E1Bound}|E_1|\leq \frac{1}{4}\left|-u_{xt}+\left(H_p+\frac{1}{2}mH_{mp}\right)\right|^2 + C|p|^{2\gamma+1}.\end{equation}
Next, to estimate $|E_2|$, we use (\ref{uxx bd}), (\ref{eq:Hpp bd}) (\ref{eq:Hppp bd}), (\ref{eq:Hm growth}), (\ref{eq:Hm bd}), and (\ref{eq:Hmpp bd}) to obtain
\begin{equation} \label{E2Bound}|E_2|\leq C|p|^{2\gamma+1}.\end{equation}
For $E_3$, we have
\begin{multline}|E_3|
\leq \frac{1}{4}\left|-u_{xt}+\left(H_p+\frac{1}{2}mH_{mp}\right)u_{xx}\right|^2+ \frac{Ck^2}{H_m^2}\left(H_{pm}^2+m^2H_{mmp}^2+H_{mp}^2\right)|-\tilde{u}_t+H_p u_x|^2, 
\end{multline}
and therefore, in view of (\ref{eq:Hpp bd}) and (\ref{eq:Hm bd}), as well as the HJ equation, we obtain
\begin{equation} \label{E3Bound}|E_3|\leq \frac{1}{4}\left|-u_{xt}+(H_p+\frac{1}{2}mH_{mp})u_{xx}\right|^2+C|p|^{2\gamma+1}. \end{equation}
Finally, for $E_4$, we observe that (\ref{uxx bd}), \eqref{eq:Hpp bd}, (\ref{eq:Hm bd}), and (\ref{eq:Hmpp bd}) yield
\begin{equation} \label{E4Bound}|E_4|\leq C|p|^{2\gamma+1}.\end{equation}
Now, \hyperref[eq:ellipticity]{(E)} implies that
\begin{multline}
    \left|-\tilde{u}_{t}+(H_p+\frac{1}{2}mH_{mp})u_{x}\right|^2-\frac{1}{4} m^2H_{mp}^2p^2+mH_mH_{pp} p^2 \\
    \geq \left|-\tilde{u}_{t}+(H_p+\frac{1}{2}mH_{mp})u_{x}\right|^2+\frac{1}{4C_0} m^2H_{mp}^2p^2=\left|-\tilde{u}_{t}+\left(H_pu_x+\frac{1}{2}mH_{mp}\right)p\right|^2\\
    +\frac{1}{C_0} \left(\frac{1}{2}mH_{mp}p\right)^2
    \geq \frac{1}{2}\left|-\tilde{u}_{t}+\left(H_p+\frac{1}{2}mH_{mp}\right)u_{x}\right|^2 + \frac{1}{C} |-\tilde{u}_{t}+H_p u_{x}|^2.  \end{multline}
So, as a result of (\ref{gradientproof2}), \eqref{eq:Hpp bd} and (\ref{eq:H superlinear growth}), we get
\begin{equation} \label{gradientproof3} L_u \left( k \frac{1}{2}\tilde{u}^2 \right) \leq  -\frac{1}{2}k\left|-\tilde{u}_{t}+\left(H_p+\frac{1}{2}mH_{mp}\right)u_{x}\right|^2 - \frac{1}{C} |p|^{2\gamma+\frac{3}{2}} +E_1+E_2+E_3+E_4. \end{equation} 
Now, since $(x_0,t_0)$ is an interior maximum point of $v$, we have $L_u(v)\geq0$. Thus, combining (\ref{E1Bound}), (\ref{E2Bound}), (\ref{E3Bound}), (\ref{E4Bound}), (\ref{gradientproof1}) and (\ref{gradientproof3}), we conclude
\[0\leq -\frac{1}{C}|p|^{2\gamma+\frac{3}{2}}+C|p|^{2\gamma+1},\]
which implies
\[|p|\leq C.\]
\end{proof}

\subsection{Estimates for MFG with \texorpdfstring{$\epsilon$}---penalized terminal condition} \label{SubsectionEpsilonPenalized}
 In order to obtain classical solutions to (\ref{PlanningSystem}), it will be necessary to use a natural approximation method, which was previously used in \cite{Porretta} to obtain weak solutions to the second-order planning problem. The solution will be obtained as the limit of solutions to standard MFG systems with a penalized terminal condition.  Specifically, we will need to prove estimates for solutions $(\ue,\me)$ to
\begin{equation}\tag{MFG$_{\epsilon}$}
\label{MFGepsilon}
    \begin{cases}
     -\ue _t+H(\ue_x,\me)=0 \text{ in }Q_T,\\
     \me _t-(\me H_p(\ue_x,\me))_x=0 \text{ in }Q_T,\\
      \me(x,0)=m_0(x),\,\epsilon \ue(x,T)=\me(x,T)-m_T(x) \text{ on }\partial Q_T.
    \end{cases} \end{equation}
As long as $\ue$ is bounded in $L^{\infty}(Q_T)$, the limit is expected to solve (\ref{PlanningSystem}). This estimate is obtained in the following lemma. While treating this system, we will temporarily assume that $H(0,0)$ is finite. This assumption will be removed in the proof of Theorem \ref{MainTheorem}.

\begin{lem}\label{UniformBoundsOnUepsilon}
For $\epsilon>0$, let $(\ue,\me)\in C^2(\overline{Q_T})\times C^1(\overline{Q_T})$ be a classical solution to system (\ref{MFGepsilon}), and set $c_1=\min\{\min_{\T}{m_0},\min_{\T}{m_T}\}$, $C_1=\max\{\max_{\T}{m_0},\max_{\T}{m_T}\}$. Assume that $H(0,0)<\infty$.  Then there exists a constant $C>0$, independent of $\epsilon$, such that 
\begin{equation} \label{ueBound} \|\ue \|_{L^{\infty}(\overline{Q_T})}\leq C.\end{equation}
Furthermore, for all $\epsilon<\frac{1}{C}$, we have 
\begin{equation}
    \label{LowerBoundmPlanning}
    \frac{c_1}{2}\leq \me(x,t)\leq 2C_1 \text{ for all }(x,t)\in \overline{Q_T},
\end{equation}
and
\begin{equation} \label{me(T)Bound}
   \|\me(T,\cdot)-m_T(\cdot)\|_{\infty}\leq \epsilon C. 
\end{equation}
\end{lem}
\begin{proof}
As a result of Proposition \ref{OscillationProp}, since $H(0,\min\limits_{\overline{Q}_T}\me)\leq H(0,0)$, there exists \[C=C(C_0,T,|H(0,0)|,|H(0,\max\limits_{\overline{Q}_T}\me)|,\overline{C}(\max\limits_{\overline{Q}_T}\me))\] such that
\[\text{osc}_{\overline{Q}_T}(\ue)\leq C.\]
To make this bound on the oscillation independent of $\epsilon$, we must obtain upper bounds on the density $\me$. Note that, from Corollary \ref{DisplacementConvexityCorBounds}, it is enough to bound $\me(T,\cdot)$ from above. To this end, let $M_0:=\max\limits_{\T}m_0$ and, for $\delta>0$, define 
\[v^{\delta}(x,t)=\ue(x,t)+H(0,M_0+\delta)(T-t).\]
Since $D^2v^{\delta}=D^2\ue$, we have that $v^{\delta}$ also solves the elliptic equation (\ref{eq:quasilinear}) in $Q_T$. Therefore, the maximum of $v^{\delta}$, must occur at $t=0$ or $t=T$. If the maximum occurred at $t=0$, then at that point
\[\ue _t -H(0,M_0)=v_t^{\delta}\leq 0,\, v_x^{\delta}=\ue_x=0,\]
and, hence,
\[0\geq \ue _t -H(0,M_0+\delta)=H(0,m_0)-H(0,M_0+\delta),\]
which is a contradiction because $H_m<0$. Therefore, for every $\delta>0$, the maximum occurs at $t=T$, and, letting $\delta\rightarrow 0$, we see that the same is true for $\delta=0$. The maximum value of $v(x,t):=\ue(x,t) +H(0,M_0)(T-t)$ equals the maximum of $\ue (x,T)$, since $v(x,T)=\ue(x,T)$. Letting $x_0\in \T$ be a point at which this maximum occurs, it follows that $v_t(x_0,T)\geq 0$, and therefore 
\[ H(0,\me(x_0,T))\geq H(0,M_0),\]
which implies that
\[\me(x_0,T)\leq M_0.\]
But, since 
\[\epsilon \ue(x,T)=\me(x,T)-m_T(x),\]
we obtain, for each $x\in \T$,
\[\epsilon \ue(x,T)\leq \epsilon \ue(x_0,T)= (\me(x_0,T)-m_T(x_0))\leq  (M_0-m_T(x_0)),\]
and, consequently,
\[\me(x,T)=\epsilon \ue(x,T)+m_T(x)\leq M_0+m_T(x)-m_T(x_0)\leq M_0+\text{osc}_{\T}(m_T).\]
We have thus shown that the bound on the oscillation of $\ue$ does not depend on $\epsilon$. Furthermore, since 
\[\epsilon\ue (x,T)=\me(x,T)-m_T(x),\]
and $\me(T,\cdot),m_T(\cdot)$ are both probability densities, we have $\int_{\T} \ue(\cdot,T)=0$, so there must exist some $x^{\epsilon}\in \T$ such that
\[\ue(x^{\epsilon},T)=0.\]
This implies that, for any $(x,t)\in \overline{Q}_T$, 
\[-\text{osc}_{\overline{Q}_T}(\ue)\leq \ue(x,t)-\ue(x^{\epsilon},T)\leq \text{osc}_{\overline{Q}_T}(\ue),\]
which shows (\ref{ueBound}). To prove (\ref{LowerBoundmPlanning}), we require $C$ to be large enough to satisfy $\frac{1}{C}\|\ue\|_{\infty}<\frac{1}{2}c_1$. Then for all $\epsilon<\frac{1}{C}$, we have 
\[\me(x,T)=m_T(x)+\epsilon\ue(x,T)\geq m_T(x)-\frac{1}{2}c_1\geq \frac{1}{2}c_1. \]
The upper bound for $\me(x,T)$ is obtained similarly. We now conclude by Corollary \ref{DisplacementConvexityCorBounds}, since the maxima and minima of $\me$ both occur at $t=0,t=T$. Finally, (\ref{me(T)Bound}) follows immediately from the terminal condition in (\ref{MFGepsilon}) and (\ref{ueBound}).
\end{proof}
While the usefulness of (\ref{MFGepsilon}) will mainly be as a tool to obtain existence for (\ref{PlanningSystem}), it can also be used to provide an interesting counterexample. Indeed, one should note that (\ref{MFGepsilon}) is not itself a planning problem, but rather a special case of a standard MFG system, which would fit in the framework of (\ref{MainMFGSystem}) if the terminal cost function $g$ were allowed to depend on $x$. Such terminal conditions are treated in \cite{Munoz1,Munoz2} under the blow-up assumption (\ref{BlowUpCondition}), as well as the requirement that
\[g(x,0)\,\,\text{is constant, or } \,\lim_{m\rightarrow0^+}g(x,m)=-\infty,\]
which is a slightly weaker version of (\ref{BlowUpCondition}). The following proposition illustrates the fact that, when such assumptions do not hold, the solution may fail to exist.
\begin{prop} \label{PropositionCounterexample}
Assume that $H(0,0)<\infty$, and that the condition $m_T>0$ in \eqref{eq:M1} does not hold, so that $m_T(x_0)<0$ for some $x_0\in \T$. Then there exists $C>0$ such that, for all $0<\epsilon<\frac{1}{C}$, there exists no classical solution to (\ref{MFGepsilon}).
\end{prop}
\begin{proof}
We assume, by contradiction, that there exists a decreasing sequence $\epsilon_n>0$, with $\lim\limits_{n\rightarrow\infty}\epsilon_n=0$, such that, for each positive integer $n$, there exists a solution $(u^n,m^n)$ to system ($\text{MFG}_{\epsilon_n}$). Since $H(0,0)<\infty$, the proof of Lemma \ref{UniformBoundsOnUepsilon} shows that, for some constant $C>0$ independent of $n\in \mathbb{N}$, we have $\|u^n\|_{\infty}\leq C$. However, this implies that
\[\|m^n(T,\cdot)-m_T(\cdot)\|_{\infty}\leq C\epsilon_n,\]
while $m^n(x_0,T)\geq 0>m_T(x_0)$, which is a contradiction. 
\end{proof}

We finish our estimates for the $\epsilon$--penalized problem with an analogue of Proposition \ref{GradientEstimateProp}.

\begin{lem}\label{GradientBoundsUepsilon}
For $\epsilon>0$, let $(\ue,\me)\in C^{3,\alpha}(\overline{Q_T})\times C^{2,\alpha}(\overline{Q_T})$ be a classical solution to system (\ref{MFGepsilon}), and assume that $H(0,0)<\infty$. Let $c_1$ and $C_1$ be as in Corollary \ref{DisplacementConvexityCorBounds}. There exists a constant $C>0$, independent of $\epsilon$, such that, for $\epsilon < \frac{1}{C}$, 
\[\|D\ue\|_{\infty}\leq C.\]
\end{lem}
\begin{proof}
We first observe that, by Corollary \ref{DisplacementConvexityCorBounds} and Lemma \ref{UniformBoundsOnUepsilon}, $\|\me\|_{\overline{Q}_T}$ and $\|(\me)^{-1}\|_{\overline{Q}_T}$ are bounded a priori in terms of $C_1$ and $c_1^{-1}$. The proof of Proposition \ref{GradientEstimateProp} may thus be repeated here, with Lemma \ref{UniformBoundsOnUepsilon} replacing the use of Proposition \ref{OscillationProp}, with one exception. Namely, the term $-H_m\ue_x\me_x$ in (\ref{grad1}) should be estimated as
\[-H_m\ue_x\me_x=-\epsilon H_m(\ue_x)^2-H_m\ue (m_T)_x\geq -H_m\ue (m_T)_x,\]
which, in view of (\ref{grad2}), yields the gradient bound in the case $t_0=T$. The rest of the argument follows unchanged.
\end{proof}

\section{Existence of classical solutions}\label{SectionExistence}

In the previous sections, a priori $L^{\infty}-$bounds were obtained for $u$, $Du$, $m$, and $m^{-1}$. This is already sufficient to obtain classical solutions to (\ref{MainMFGSystem}), following the arguments of \cite{Munoz1,Munoz2}. The existence of solutions to (\ref{PlanningSystem}), on the other hand, is a more delicate issue, because the Neumann type boundary condition that appears in the linearization makes the latter non--invertible. Namely, the linearization of (\ref{eq:quasilinear}) is
\begin{equation*}
    \begin{cases}
    L_u(w)=f & \text{ in }Q_T,\\
    (-1,H_p(u_x,m))\cdot Dw=g_1(x) & \text{ at }t=0,\\
    (1,-H_p(u_x,m))\cdot Dw=g_2(x) & \text{ at }t=T,
    \end{cases}
\end{equation*}
which is an oblique boundary value problem that is only solvable for certain functions $f,\,g_1,\,g_2$ satisfying a compatibility condition that itself depends on $u$. This failure of invertibility precludes the direct use of the implicit function theorem and thus of the method of continuity, which means a different approach is needed.
Indeed, we will obtain the solution as the limit as $\epsilon \rightarrow 0$ of the solution to the $\epsilon$--penalized problem (\ref{MFGepsilon}). We begin by noting, in the following lemma, that for $\epsilon$ small enough, the solutions to (\ref{MFGepsilon}) are a priori uniformly bounded in $C^{1,\beta}(\overline{Q}_T)$, for some $0<\beta<1$, and that the system thus has a classical solution.

\begin{lem}\label{HolderGradientUepsilon}
Let $C$ be as in Lemma \ref{UniformBoundsOnUepsilon}. For all $0<\epsilon<\frac{1}{C}$, (\ref{MFGepsilon}) has a unique smooth solution $(\ue,\me)\in C^{3,\alpha}(\overline{Q_T})\times C^{2,\alpha}(\overline{Q_T})$. Moreover, there exist constants $K>0$, $0<\beta<1$, independent of $\epsilon$, such that
\begin{equation} \label{uepsilonHolderEstimate}\|\ue\|_{C^{1,\beta}}\leq K.\end{equation}

\end{lem}
\begin{proof}
 The a priori $C^1-$bounds on $\ue$, as well as $L^{\infty}-$bounds on $\me$ and $(\me)^{-1}$ (and thus on the ellipticity constants of the system), were all established in Lemmas \ref{UniformBoundsOnUepsilon} and \ref{GradientBoundsUepsilon}. The H\"older estimate for the gradient then follows in the same way as in \cite[Lem. 4.1]{Munoz1}, by directly applying the classical $C^{1,\alpha}$--estimates for quasilinear elliptic equations with oblique boundary conditions (see \cite[Lem. 2.3]{Lieberman}). Indeed, it suffices to verify that, for $(x,t,z,p,s)\in \T \times \{0,T\} \times \R \times \R \times \R $, the boundary condition
\[B^{\epsilon}(x,0,z,p,s)=-s+H(p,m_0(x)),\,\,\,B^{\epsilon}(x,T,z,p,s)=s-H(p,\epsilon z+m_T(x)),\]
is oblique. For this purpose, we let $\nu(x,t)$ denote the outward unit normal vector at $(x,t)\in \partial Q_T$. Then we have
\[ D_{(p,s)}B^{\epsilon}(x,0,z,p,s)\cdot \nu(x,0)=-B_s^{\epsilon}(x,0,z,p,s)=1>0,\]
\[ D_{(p,s)}B^{\epsilon}(x,T,z,p,s)\cdot \nu(x,T)=-B_s^{\epsilon}(x,T,z,p,s)=1>0\]
and thus the a priori estimate (\ref{uepsilonHolderEstimate}) follows.
The proof of existence is then the same as in \cite[Thm. 1.1]{Munoz1} through the method of continuity.
\end{proof}

We now have enough information on the $\epsilon$--penalized problem to prove our first theorem.

\begin{proof}[Proof of Theorem \ref{MainTheorem}]
 We initially assume that $m_0,m_T\in C^{\infty}(\T)$. The proof of part (ii), corresponding to (\ref{MainMFGSystem}), is identical to the one carried out in \cite[Thm. 1.1]{Munoz1}. We simply note that the condition $\lim\limits_{m\rightarrow0^+}H(p,m)=+\infty$ in that proof was only used to guarantee the existence of a positive lower bound for the density, which in turn makes the equation (\ref{eq:quasilinear}) uniformly elliptic. In our case, the lower bound is a consequence of Corollary \ref{DisplacementConvexityCorBounds} and Proposition \ref{mAnduBounds}.

Now, for the case of (\ref{PlanningSystem}), we remark first that uniqueness of $u$, up to a constant, follows by the standard Lasry-Lions monotonicity method. To establish existence, we consider first the approximate system (\ref{MFGepsilon}), under the assumption $H(0,0)<\infty$. We assume that $\epsilon>0$ is small enough for Lemma \ref{HolderGradientUepsilon} to guarantee the existence of solutions $(\ue,\me)$. Letting $0<\beta<1$ be as in Lemma \ref{HolderGradientUepsilon}, we also have \eqref{uepsilonHolderEstimate}, for some constant $K>0$ independent of $\epsilon$.
We infer that there exist a subsequence $\{u_n\}_n\subset \{\ue\}_{\epsilon}$, and $u\in C^{1,\alpha}(\overline{Q_T})$, such that $u_n \rightarrow u$ uniformly. Furthermore, in view of Lemma \ref{UniformBoundsOnUepsilon}, there exists $C>0$, independent of $\epsilon$, such that
\[\frac{1}{C}\leq \me(x,t)\leq C\text{ for all }(x,t)\in \overline{Q_T}.\]
We let $(A,B)$ and $(A_n,B_n)$, be the quasilinear operators and boundary conditions corresponding, respectively, to $u$ and $u_n$. Then one has
\[(A_n,B_n)\rightarrow (A,B) \text{ locally uniformly,}\]
\[D_qB_n\cdot \nu= 1.\]
Hence, by Fiorenza's convergence theorem for elliptic equations with oblique boundary conditions (see \cite[Thm. 2.5]{Munoz1}, \cite[Chapter 17, Lemma 17.29]{GilbargTrudinger}), we obtain $u_n\rightarrow u$ in $C^{2,\alpha}(\overline{Q}_T)$, and $u$ solves \eqref{eq:quasilinear}, with the boundary condition corresponding to (\ref{PlanningSystem}). The $C^{3,\alpha}$ regularity (and, in fact, uniform convergence in $C^{3,\alpha}$) then follows readily from the standard Schauder estimates for linear oblique problems, as in \cite[Thm. 1.1]{Munoz1}.

The last step will be to remove the assumption that $m_0\in C^{\infty}(\T)$ and, for (\ref{PlanningSystem}), the assumptions that $m_T\in C^{\infty}(\T)$ and $H(0,0)<\infty$. We will explain the argument for \eqref{PlanningSystem}, with the treatment of \eqref{MainMFGSystem} being completely analogous.  Consider, for $\delta>0$, the modified Hamiltonians $H^{\delta}(p,m):=H(p,m+\delta)$, which satisfy \hyperref[eq:Hpp bd]{(H)} and \hyperref[eq:ellipticity]{(E)}, uniformly in $\delta$, as well as $H^{\delta}(0,0)<\infty$, and a sequence of $C^{\infty}$ densities $(m_0^{\delta},m_T^{\delta})$, uniformly bounded in $C^{2,\alpha}$ and bounded away from $0$, converging uniformly to $(m_0,m_T)$.  Let  $(u^{\delta},m^{\delta})$ be the corresponding solutions to 
 \begin{equation}
     \begin{cases}
      -u_t^{\delta}+H^{\delta}(u_x^{\delta},m^{\delta})=0 & \text{ in }Q_T,\\
      \int_0^T\intT u^{\delta} =0,\\
      m_t^{\delta}-(m^{\delta}H_p^{\delta}(u_x^{\delta},m^{\delta}))_x=0 & \text{ in }Q_T,\\
      m^{\delta}(\cdot,0)=m_0^{\delta},\,\, m^{\delta}(\cdot,T)=m_T^{\delta} & \text{ on } \T.
     \end{cases}
 \end{equation}
 Propositions \ref{GradientEstimateProp} and \ref{OscillationProp}, and Corollary \ref{DisplacementConvexityCorBounds}, yield uniform $C^1-$bounds on $u^{\delta}$, and thus, as in the proof of Lemma \ref{uepsilonHolderEstimate}, uniform $C^{1,\beta}$ bounds for some $0<\beta<1$. We may thus conclude by letting $\delta \rightarrow 0$ and applying Fiorenza's convergence result as above.
\end{proof}

\section{Regularity of weak solutions} \label{WeakSection} 
We now study the existence and regularity of solutions to (\ref{MainMFGSystem}) and (\ref{PlanningSystem}) under the weaker assumption that, for some $\kappa>0$
\[\intT \frac{1}{m_0^{\kappa}(x)}dx<\infty,\,\,\intT \frac{1}{m_T^{\kappa}(x)}dx<\infty. \]
We note that, in particular, the above conditions allow for the densities to vanish at a set of measure zero. This, in general, creates significant issues, because \eqref{eq:quasilinear} is no longer uniformly elliptic. The key estimate that will allow us to prove smoothness in this setting is an interior lower bound on the density which depends only on $t^{-1}$, $\|m_0^{-\kappa}\|_1$ (and $(T-t)^{-1},\|m_T^{-\kappa}\|_1$, in the case of \eqref{PlanningSystem}). Indeed, this yields uniform ellipticity of \eqref{eq:quasilinear} away from $t=0$ and $t=T$.

We begin by giving the standard definition of a weak solution (see, for instance, \cite{CardaliaguetGraber, Munoz1, Porretta2}).

\begin{defn}\label{def:weaksol_def}[Definition of weak solution]
A pair $(u,m)\in\text{{BV}}(Q_{T})\times L_{+}^{\infty}(Q_{T})$
is called a weak solution to (\ref{MainMFGSystem}) (respectively (\ref{PlanningSystem})) if the following conditions
hold:

\begin{enumerate}
\item[(i)] $u_x\in L^{2}(Q_{T}),u\in L^{\infty}(Q_{T}),$ $m\in C^{0}([0,T];H^{-1}(\T))$.
\item[(ii)] $u$ satisfies the HJ inequality
\[
-u_{t}+H(u_x,m)\leq 0\;\;\text{ in }Q_{T},\]
in the distributional sense.
\item[(iii)] $m$ satisfies the continuity equation
\begin{equation}
m_{t}-(mH_p(u_x,m))_x=0\text{ in }Q_{T},\label{eq: FP weak}
\end{equation}
in the distributional sense.
\item[(iv)] We have $m(\cdot,T)\in L^\infty(\T)$. Moreover, $m(\cdot,0)=m_0$ in $H^{-1}(\T)$ and $u(\cdot,T)=g(m(\cdot,T))$ in the sense of traces (respectively, $m(\cdot,T)=m_T$ in $H^{-1}(\T)$).
\item[(v)] The following identity holds:
\[
\int \int_{Q_T} m(x,t)(H(u_x,m)-H_p(u_x,m)u_x)dxdt
=\int_{\T}(m(x,T)u(x,T)-m_{0}(x)u(x,0))dx.\label{eq:identity weak}
\]
\end{enumerate}
\end{defn}
The following lemma will be needed to show that, for solutions to \eqref{MainMFGSystem}, our interior regularity results may be extended up to time $t=T$.
\begin{lem}\label{DecreasingLemmaForDisplacement}
Let $(u,m)$ be a smooth solution to (\ref{MainMFGSystem}) under the assumptions of Theorem \ref{MainTheorem} and assume that \eqref{RadialConditionH} holds. Then, for every convex function $h\in C^2(0,\infty),$ the map
\[t\rightarrow \intT h(m(x,t))dx\]
is decreasing. Moreover, there exists a constant $C=C(C_0,\|(g')\|_{L^{\infty}([\min m_0,\max m_0])}^{-(\gamma-1)})$ such that
\[ \frac{d}{dt}\int_{\T} h(m(x,T))dx +\frac{1}{C}\int_{\T}h''(m(x,T))  |m_x(x,T)|^{\gamma} \leq 0.\]
\end{lem}
\begin{proof}
In view of Proposition \ref{DisplacementProp}, we have that 
\[\frac{d^2}{dt^2}\intT h(m(x,t))dx\geq 0,\]
and, thus, the function 
\[d(t):= \frac{d}{dt} \intT h(m(x,t))dx\]
is increasing. We then infer that the monotonicity will follow if we show that 
\[d(T)\leq 0.\]
Since $u(\cdot,T)=g(m(\cdot,T)),$ and $m$ satisfies the continuity equation, we have 
\begin{align*}
    d(T)&=\int_{\T} h'(m(x,T)) m_t(x,T)dx=\int_{\T} h'(m)(mH_p(u_x,m))_xdx=-\int_{\T}h''(m) m_x H_p(m_xg'(m),m).\end{align*}
Now, as a result of \eqref{RadialConditionH} and \eqref{eq:Hpp bd},
\[H_p(m_xg'(m),m)(m_x g'(m))\geq \frac{1}{C}|m_xg'(m)|^{\gamma},\]
and, therefore,
\[d(T)\leq -\frac{1}{C}\int_{\T}h''(m)  |m_x|^{\gamma}.\]
\end{proof}
We are now ready to obtain the interior lower bounds on $m$. Our method of proof relies on the displacement convexity formula \eqref{DisplacementFormula}, and uses similar techniques to \cite[Prop. 5.2]{Porretta2}.
\begin{prop}\label{LocalLowerBoundDensityProp}
Let $(u,m)$ be a smooth solution to (\ref{MainMFGSystem}) or (\ref{PlanningSystem}), under the same assumptions as in Theorem \ref{MainTheorem}. Assume, furthermore, that \eqref{eq:Hm growth weak} holds and, in the case of (\ref{MainMFGSystem}), assume that \eqref{RadialConditionH} holds. Let
\[\beta=\frac{2}{\kappa -s-1},\] and let $\delta>0$. Then, there exist a constant $C=C(C_0\|m_0^{-\kappa }\|_{L^1},\|m_T^{-\kappa }\|_{L^1},\delta^{-1})$ such that
\begin{equation}\label{eq: MFGP local lower bound} m(x,t)\geq \frac{1}{C}\left(\frac{1}{t^{\beta+\delta}}+\frac{1}{(T-t)^{\beta+\delta}}\right)^{-1}.\end{equation}
Furthermore, in the case of (\ref{MainMFGSystem}), one has
\begin{equation} \label{eq: MFG local lower bound}m(x,t)\geq \frac{1}{C}t^{\beta+\delta}.\end{equation}
\end{prop}
\begin{proof}
Using the displacement convexity formula (\ref{DisplacementFormula}) for $h(m)=\frac{1}{m^{\kappa }}$, we have, for each $t\in[0,T]$,
\begin{equation}\label{m-1bound}
\intT \frac{1}{m^\kappa (x,t)}dx \leq  \max \left( \intT \frac{1}{m_0^\kappa (x)}dx, \intT \frac{1}{m^\kappa (x,T)}dx\right).
\end{equation}
Combined with Lemma \ref{DecreasingLemmaForDisplacement} (for the case of \eqref{MainMFGSystem} where $m(\cdot,T)$ is not prescribed), this yields
\begin{equation}\label{Sup_1/m}
    \sup\limits_{t\in [0,T]}\|m^{-\kappa}(t)\|_1\leq C.
\end{equation}
Next, for any $p>1$, we define the function 
\[\phi(t):=\intT m^{-p\kappa }(t)dx.\]
Using Proposition \ref{DisplacementProp} with $h(m)=m^{-p\kappa },$ as a result of \hyperref[eq:ellipticity]{(E)}, we obtain 
\begin{align*} 
    \frac{d^2}{dt^2}\intT \frac{m^{-p\kappa }(t)}{p\kappa (p\kappa +1)}dx &\geq -\frac{1}{C}\intT m^{-p\kappa -1}mH_{pp}H_m(m_x)^2dx\geq \intT \frac{1}{C}m^{-p\kappa -1+s}(m_x)^2 dx\\
    &\geq \frac{1}{C(\frac{-p\kappa +s +1}{2})^2}\intT (m^{\frac{-p\kappa +s+1}{2}})_x^2dx. 
\end{align*}
As a result, letting 
\[C_{p}:=\frac{C(p\kappa -s-1)^2}{4p\kappa(p\kappa +1)},\]
\[ \label{eq:Lambda def}\lambda :=\frac{-p\kappa +s+1}{2},\]
we have shown that 
\begin{equation}\label{eq:phi1}
    C_{p}\phi''(t)\geq \intT (m^{\lambda})_x^2 dx.
\end{equation}
From \hyperref[eq:mweak]{(W)}, and the fact that $p>1$, we see that $\lambda<0$. For each $t\in [0,T]$, since $m(\cdot,t)$ is a probability measure, there exists a point $x_0^t$ such that $m(x_0^t,t)=1$. By the fundamental theorem of calculus,
\begin{equation} \label{eq: FTC appli}    
\Big\| m^{\lambda }(t)-1 \Big\|_{\infty}^2=\Big\| m^{\lambda }(t)-m(x_0^t,t)^{\lambda} \Big\|_{\infty}^2\leq C \intT (m^{\lambda})_x^2 dx,\end{equation}
and therefore
\begin{equation} \label{eq:Morrey} \Big\|\frac{1}{m}\Big\|_{\infty}^{2|\lambda|}\leq C\Big(\int_{\T} (m^{\lambda})_x^2 dx+1\Big). \end{equation}
Now, using \eqref{Sup_1/m}, we obtain
\[\phi=\int_{\T} \frac{1}{m^{\kappa p}}\leq \int_{\T}\frac{1}{m^\kappa} \left\| \frac{1}{m}\right\|_{\infty}^{\kappa(p-1)}\leq C\left\| \frac{1}{m}\right\|_{\infty}^{\kappa(p-1)},\]
and, consequently,
\begin{equation}\label{eq:InterpolationForphi}
    C^{-r}\phi^r\leq \left\| \frac{1}{m}\right\|_{\infty}^{2|\lambda|},
\end{equation}
where $r:=\frac{2|\lambda|}{\kappa (p-1)}$. From condition \hyperref[eq:mweak]{(W)}, we see that $r>1$. Combining \eqref{eq:phi1}, \eqref{eq:Morrey}, and \eqref{eq:InterpolationForphi}, we obtain
\[C_{p}\Big(\phi''(t)+1\Big)-C^{-r}\phi(t)^{r} \geq 0,\]
that is, for some constant $C=C(p)$,
\begin{equation}\label{FinalPDEForPhi}
-\phi''(t)+\frac{1}{C}\phi^{r}\leq C.
\end{equation}
A straightforward computation then shows that the functions
\[\psi_1(t)=A_pt^{-p \kappa \beta}+K_p,\]
\[\psi_2(t)=A_p(T-t)^{-p\kappa  \beta}+K_p,\]
\[\psi(t)=\psi_1(t)+\psi_2(t),\]
are supersolutions of \eqref{FinalPDEForPhi} for large enough $A_p,K_p$. Therefore, we have
\begin{equation} \label{eq:lowBound0}\int_{\T} m^{-p \kappa}(t) \leq A_p(t^{-p \kappa \beta}+(T-t)^{-p \kappa \beta}) +2K_p. \end{equation}
Now, going back to \eqref{eq:phi1} and \eqref{eq:Morrey}, we may write
\begin{equation} \label{eq:lowBound1}\left\|\frac{1}{m} \right\|_{\infty}^{2 |\lambda|}(t) \leq C(\frac{d^2}{dt^2} \int_{\T} m^{-p\kappa}+ 1).  \end{equation}
In view of (\ref{DisplacementFormula}), for $q>0$, the map
\begin{equation} \label{eq:lowBound convex map} t \mapsto \int_{\T} m^{-q}(t)\end{equation}
is convex in $[0,T]$. Thus, fixing $t_0\in (0,\frac{T}{2}]$, we infer that, for each $t\in [t_0,T-t_0],$ 
\begin{multline*}
    \left(\int_{\T} m^{-2|\lambda| q}(t) \right)^{\frac{1}{q}}\leq \frac{2}{t_0}\max\left(\int_{\frac{t_0}{2}}^{t_0}\left(\int_\T m^{-2|\lambda|q} \right)^{\frac{1}{q}},\int_{T-t_0}^{T-\frac{t_0}{2}}\left(\int_\T m^{-2|\lambda|q}\right)^{\frac{1}{q}} \right)\\
    \leq\frac{2}{t_0}\int_{\frac{t_0}{2}}^{T-\frac{t_0}{2}}\left(\int_\T m^{-2|\lambda|q}\right)^{\frac{1}{q}}.
\end{multline*}
Letting $q \rightarrow \infty $, we obtain
\begin{equation} \label{eq:local MFGP}\left\| m^{-1} \right\|_{L^{\infty}(\T \times [t_0,T-t_0])}^{2|\lambda|}\leq \frac{2}{t_0} \int_{\frac{t_0}{2}}^{T-\frac{t_0}{2}}  \left\| m^{-1}(t) \right\|_{\infty}^{2|\lambda|}dt.\end{equation}
 Now, letting $\zeta \in C^{\infty}(Q_T)$ be a test function, supported in $[\frac{t_0}{4},T-\frac{t_0}{4}]$, such that $0\leq \zeta \leq 1$,  $\zeta \equiv 1$ in $[\frac{t_0}{2},T-\frac{t_0}{2}]$, and $\int_{0}^T |\zeta''(t)|dt\leq \frac{C}{t_0}$, we see that \eqref{eq:local MFGP} implies
\begin{equation} \label{eq: convexity argu} \left\| m^{-1} \right\|_{L^{\infty}(\T \times [t_0,T-t_0])}^{2|\lambda|}\leq \frac{2}{t_0} \int_{0}^{T}   \left\| m^{-1} \right\|_{\infty}^{2|\lambda|} (t)\zeta(t)dt.   \end{equation}
Hence, integrating by parts twice, we infer from \eqref{eq:lowBound0} and \eqref{eq:lowBound1} that
\[\left\| m^{-1} \right\|_{L^{\infty}(\T \times [t_0,T-t_0])}^{2|\lambda|} \leq \frac{C}{t_0}\left(\int_{0}^{T} \int_{\T} (m^{-p\kappa} \zeta '') + CT \right) \leq C \left(\frac{1}{t_0^{2+p\kappa \beta}}+\frac{1}{t_0}\right),\]
which yields
\[\left\| m^{-1} \right\|_{L^{\infty}(\T \times [t_0,T-t_0])}\leq C\left(\frac{1}{t_0^{\frac{2+p\kappa \beta}{ 2|\lambda|}}}+\frac{1}{t_0^{\frac{1}{2|\lambda|}}}\right).\]
Now, recalling \eqref{eq:Lambda def}, we see that
\[\label{eq:lowbd pf 1}\lim_{p\rightarrow \infty} \frac{1}{2|\lambda|}=0 \text{ and }\lim_{p\rightarrow \infty} \frac{2+p\kappa \beta}{ 2|\lambda|}=\beta.\]
Thus, we may fix $p$ chosen large enough that $\frac{2+\kappa \beta}{ 2|\lambda|}<\beta+\delta$, and, as a result of \eqref{eq:lowbd pf 1},
\[\left\| m^{-1} \right\|_{L^{\infty}(\T \times [t_0,T-t_0])}\leq C\frac{1}{t_0^{\beta+\delta}}.\]
This implies \eqref{eq: MFGP local lower bound}. Now, for the case of \eqref{MainMFGSystem}, we simply observe that, from Lemma \ref{DecreasingLemmaForDisplacement}, the map \eqref{eq:lowBound convex map} is non-increasing on $[0,T]$, and, thus, \eqref{eq:local MFGP} may be strengthened to
\[\left\| m^{-1} \right\|_{L^{\infty}(\T \times [t_0,T])}^{2|\lambda|} \leq \frac{2}{t_0} \int_{\frac{t_0}{2}}^{T}  \left\| m^{-1} \right\|_{\infty}^{2|\lambda|}(t)dt. \]\end{proof}
The following lemma is a basic computation exploiting \eqref{eq:ellipticity}, and will be used in the proof of Theorem \ref{LocalTheorem} to estimate the terms arising from the Lasry-Lions monotonicity method.
\begin{lem}\label{EstimateOfHForLasryLions}There exists a constant $C=C(C_0)>0$ such that, given $-\infty<p_0<p_1<\infty$ and $0<m_0<m_1<\infty$, we have
\begin{multline} \label{eq: LasryLions1}\left(m_1 H_p(p_1,m_1)-m_0 H_p(p_0,m_0)\right)(p_1-p_0)-\left(H(p_1,m_1)-H(p_0,m_0)\right)(m_1-m_0) \\
\geq \frac{m_1+m_0}{C}(p_1-p_0)^2+\frac{k}{C}(m_1-m_0)^2,\end{multline}
where $k=\min_{[p_0,p_1]\times[m_0,m_1]} (-H_m(p,m))$. Moreover, if $H$ satisfies \hyperref[eq:mweak]{(W)}, then 
\begin{multline}\label{eq: LasryLions2} \left(m_1 H_p(p_1,m_1)-m_0 H_p(p_0,m_0)\right)(p_1-p_0)-\left(H(p_1,m_1)-H(p_0,m_0)\right)(m_1-m_0) \\
\geq \frac{m_1+m_0}{C}(p_1-p_0)^2+\frac{1}{C(s+1)}(m_1^{s+1}-m_0^{s+1})(m_1-m_0).\end{multline}
\end{lem}

\begin{proof}
Following the technique carried out in \cite{Takis},
for $z\in [0,1],$ we define 
\[\Delta p=p_1-p_0\,\,\Delta m=m_1-m_0,\, p_z=p_0+z\Delta p,\,m_s=m_0+z\Delta m.\]
We then let
\[\phi(z)=(m_zH_p(p_z,m_z)-m_0H_p(p_0,m_0))\Delta p-(H(p_z,m_z)-H(p_0,m_0))\Delta m,\]
and differentiation yields
\[\phi'(z)=m_z H_{pp} (\Delta p)^2+m_zH_{mp}\Delta m \Delta p - H_m(\Delta m)^2.\]
Now, in view of \eqref{eq:ellipticity}, we have, for some constant $C>0$,
\[-H_m\geq \frac{1}{4 H_{pp}}m_zH_{mp}^2(1+\frac{1}{C})-\frac{1}{C}H_m.\]
Therefore,
\begin{multline}\label{eq:LL}
\phi'(z)\geq m_z\left(\frac{1}{\sqrt{1+\frac{1}{C}}}\sqrt{H_{pp}}\Delta p+ \frac{\sqrt{1+\frac{1}{C}}}{2 \sqrt{H_{pp}}}H_{mp}\Delta m\right)^2\\
+m_z H_{pp} (\Delta p)^2 (1-\frac{1}{1+\frac{1}{C}})-\frac{1}{C}H_m(\Delta m)^2.\end{multline}
If \hyperref[eq:mweak]{(W)} holds, then, up to increasing the constant $C>0$, as well as using (\ref{eq:Hpp bd}) and \hyperref[eq:mweak]{(W)}, we obtain
\[\phi'(z)\geq \frac{1}{C}(m_z(\Delta p)^2+m_z^{s} (\Delta m)^2),\]
and integrating over $[0,1]$ then yields \eqref{eq: LasryLions2}. The proof of \eqref{eq: LasryLions1} follows from \eqref{eq:LL} in the same way. \end{proof}

Before proving Theorem \ref{LocalTheorem}, we remind the reader that assumption \hyperref[eq:M1]{(M)} will not be in place, and will be instead replaced by \hyperref[eq:mweak]{(W)}.
\begin{proof}[Proof of Theorem \ref{LocalTheorem}]
For $\epsilon \in (0,1)$, let $\me_0$, $\me_T$ be smooth, positive densities such that, for $\theta\in \{0,T\}$, 
\[\me_{\theta}\rightarrow m_{\theta}\, \text{ a.e. in}\, \T, \, ||\me_{\theta}||_{\infty}\leq C \, \text{ and } \, ||(\me_{\theta})^{-\kappa}||_{1}\leq C,\] where $C>0$ is a constant  independent of $\epsilon$. Let $(u^{\epsilon,1},m^{\epsilon,1})$ be a smooth solution to (\ref{PlanningSystem}) obtained from taking $\me_0$ and $\me_T$, respectively, as the initial and terminal densities. Similarly, let $(u^{\epsilon,2},m^{\epsilon,2})$ be the smooth solution to (\ref{MainMFGSystem}) corresponding to the initial density $\me_0$.  The existence and regularity of such solutions is guaranteed by Theorem \ref{MainTheorem}. We may further choose the $u^{\epsilon,1}$ to be normalized so that $\int_{\T}u^{\epsilon,1}(T)=0$.

As in the proof of Proposition \ref{LocalLowerBoundDensityProp}, we obtain, for some $C>0$ independent of $\epsilon$ and for $i\in\{1,2\}$,
\begin{equation} \label{eq:1/mL1bound} \|(m^{\epsilon,i})^{-\kappa}\|_{1}\leq C.\end{equation}
On the other hand, Corollary \ref{DisplacementConvexityCorBounds} and Proposition \ref{mAnduBounds} yield
\begin{equation} \label{eq:LocalmEpsilonUpperBound} \|m^{\epsilon,i}\|_{\infty}\leq C,\end{equation}
and \eqref{eq:LocalmEpsilonUpperBound}, \eqref{eq:Hm growth weak} and Proposition \ref{LocalLowerBoundDensityProp} imply that
\begin{equation}
    \int_{0}^{T}|H(0,\min_{\T} m^{\epsilon,i}(s)|ds\leq C.
\end{equation}
Thus, as a result of \eqref{eq:g(0)condition}, Proposition \ref{mAnduBounds}, and Proposition \ref{OscillationProp},
\begin{equation} \label{eq:LocaluEpsilonBounds} \|u^{\epsilon,i}\|_{\infty}\leq C. \end{equation}
We will first observe that, up to a subsequence, there is convergence to a weak solution. Indeed, given $0<\epsilon,\epsilon'<1$, applying the Lasry-Lions monotonicity method to the corresponding systems yields, for $i\in\{1,2\}$,

\begin{multline}
    \int_{\T} (u^{\epsilon,i}(T)-u^{\epsilon',i}(T))(m^{\epsilon,i}(T)-m^{\epsilon',i}(T))-\int_{\T} (u^{\epsilon,i}(0)-u^{\epsilon',i}(0))(m^{\epsilon,i}(0)-m^{\epsilon',i}(0)) \\
    +\int\int_{Q_T} \left(m^{\epsilon,i} H_p(u^{\epsilon,i}_x,m^{\epsilon,i})-m^{\epsilon',i} H_p(u^{\epsilon',i}_x,m^{\epsilon',i})\right)(u^{\epsilon,i}_x-u^{\epsilon',i}_x)\\
    -\left(H(u^{\epsilon,i}_x,m^{\epsilon,i})-H(u^{\epsilon',i}_x,m^{\epsilon',i})\right)(m^{\epsilon,i}-m^{\epsilon',i})=0.
\end{multline}
Lemma \ref{EstimateOfHForLasryLions} therefore yields
\begin{multline}
    \int_{\T} (u^{\epsilon,i}(T)-u^{\epsilon',i}(T))(m^{\epsilon,i}(T)-m^{\epsilon',i}(T))-\int_{\T} (u^{\epsilon,i}(0)-u^{\epsilon',i}(0))(m^{\epsilon,i}(0)-m^{\epsilon',i}(0)) \\
   + \int\int_{Q_T} \left(\frac{m^{\epsilon,i}+m^{\epsilon',i}}{C}(u^{\epsilon,i}_x-u^{\epsilon',i}_x)^2+\frac{1}{C(s+1)}((m^{\epsilon,i})^{s+1}-(m^{\epsilon',i})^{s+1})(m^{\epsilon,i}-m^{\epsilon',i}) \right)\leq 0.
\end{multline}
Proceeding as in \cite[Thm. 1.2]{Munoz1}, it readily follows that, for $i\in\{1,2\}$, as $\epsilon \rightarrow 0$, $(u^{\epsilon,i},m^{\epsilon,i})$ converges to a weak solution $(u^i,m^i)$. 

It remains to show the interior regularity. For $\delta>0$, we define
\[I_{1,\delta}=[\delta,T-\delta],\,\, I_{2,\delta}=[\delta,T].\]
By Proposition \ref{LocalLowerBoundDensityProp}, there exists $C=C(\delta^{-1})$ such that, for $t_i \in I_{i,\delta/4}$,
\begin{equation} \label{eq:LocalmEpsilonLowerBound} m^{\epsilon,i}(\cdot,t_i)\geq \frac{1}{C}.\end{equation}
We must first obtain a priori gradient bounds for $u^{\epsilon,i}$ on $I_{i,\delta/2}$. Setting
\[\phi_1(t)=(t-\delta/4)^{-2/(\gamma-1)}+(T-\delta/4-t)^{-2/(\gamma-1)}\,
\phi_2(t)=(t-\delta/4)^{-2/(\gamma-1)}, \]
we go through the steps of Proposition \ref{GradientEstimateProp}, replacing the function $v$ by 
\[v_i(x,t)=\frac{1}{2}(u^{\epsilon,i}_x)^2+\frac{1}{2}(\tilde{u}^{\epsilon,i})^2-K\phi_i(t),\]
where $K>0$, $\tilde{u}^{\epsilon,i}$ is defined as in Proposition \ref{GradientEstimateProp}. We consider the maximum point $(x_0,t_0)$ of $v_i$ in $\T \times I_{i,\delta/4}$. In the case of (\ref{PlanningSystem}), namely $i=1$, this maximum must be attained in the interior of $I_i$, since $\phi_i$ is unbounded near the endpoints. When $i=2$, the maximum may be attained at $t=T$, and the proof that $|p|\leq C$ in this case follows through unchanged from Case 1 of Proposition \ref{GradientEstimateProp}. If the maximum is achieved at an interior time, the steps of Proposition \ref{GradientEstimateProp} yield that if $v_i(x_0,t_0)$ is large enough, then
\[0\leq -|p|^{2\gamma}+|p|^{2\gamma-2}-K(-\phi_i''+\frac{1}{C}K^{\gamma}\phi_i^\gamma-C\phi_i).\]
Similarly to Proposition \ref{LocalLowerBoundDensityProp}, we see that, if $K$ is chosen large enough, $\phi_i$ must be a supersolution to
\[-\phi_i''+\frac{1}{C}K^{\gamma}\phi_i^\gamma-C\phi_i=0,\]
which then implies $p\leq C$, and thus $|u^{\epsilon,i}_x|$ is bounded on $I_{i,\delta/2}$. In view of \eqref{eq:LocalmEpsilonLowerBound} and \eqref{eq:LocalmEpsilonUpperBound}, $|u^{\epsilon,i}_t|=|H(u^{\epsilon,i}_x,m^{\epsilon,i})|$ is also bounded on $I_{i,\delta/2}$. That is, we have
\begin{equation} \label{LocalGradientBound}
    \|u^{\epsilon,i}\|_{C^1(\T \times I_{i,\delta/2})}\leq C.
\end{equation}
The interior $C^{1,\alpha}$-estimates for quasilinear elliptic equations (see \cite[Chapter 13, Thm. 13.6]{GilbargTrudinger}), followed by the interior Schauder estimates (see \cite[Chapter, 2, (1.12)]{Ladyzhenskaya}) then yield, for some $C=C(\delta^{-1})$, and for $i\in\{1,2\},$
\begin{equation}\label{interiorC3bound}||u^{\epsilon,i}||_{C^{3+\alpha}(\T \times I_{1,\delta})}\leq C.
\end{equation}For $i=1$, by virtue of the Arzelà--Ascoli theorem, we may finish the proof by simply letting $\epsilon \rightarrow 0$. On the other hand, for $i=2$ (that is, the case of \eqref{MainMFGSystem}), we require estimates up to the terminal time $T$. We first observe that \eqref{eq:LocalmEpsilonLowerBound}, \eqref{eq:LocalmEpsilonUpperBound}, and \eqref{interiorC3bound}  imply  that $u^{\epsilon,2}$ solves, in $I_{2,\delta}\times \T,$ a system of the form \eqref{MainMFGSystem}, where the initial density $m^{\epsilon,2}(\cdot,\delta)$ is bounded below by a positive constant, and bounded above in $C^{2,\alpha}(\T)$. Moreover, as in Lemma \ref{HolderGradientUepsilon}, \eqref{LocalGradientBound} implies that $u^{\epsilon,2}$ is bounded in $C^{1,\beta}$ for some $0<\beta<1$. We may now conclude through the same convergence argument as in the proof of Theorem \ref{MainTheorem}. \end{proof}

 Finally, by requiring some further regularity on the marginals, we establish additional Sobolev regularity for the weak solutions.\\
\begin{prop} \label{WeakRegularityProp}
Let $m_0,m_T$ satisfy $(m_0)_{xx}, (m_T)_{xx}\in L^1(\T)$. Let $(u,m)$ be a weak solution to (\ref{MainMFGSystem}) or (\ref{PlanningSystem}) under the assumptions of Theorem \ref{LocalTheorem}. Then, for some constant $C>0$ we have:
\begin{itemize}
\item In the case of (\ref{MainMFGSystem}),
\begin{equation}
    \intT g'(m(x,T))|m_x(x,T)|^2+\int_0^T\intT m(u_{xx})^2+m^s(m_x)^2dxdt \leq C,
\end{equation}
where $C=C(\|u\|_{\infty},\|(m_0)_{xx}\|_1,C_0)$.
\item In the case of (\ref{PlanningSystem}),
\begin{equation}
    \int_0^T\intT m(u_{xx})^2+m^s(m_x)^2dxdt \leq C,
\end{equation}
where $C=C(\|u\|_{\infty},\|(m_0)_{xx}\|_1,\|(m_T)_{xx}\|_1,C_0)$. 
\end{itemize}
\end{prop}
\begin{proof}
We will show the result in the case where $(u,m)$ is smooth, since the general case follows by considering the approximations employed in the proof of Theorem \ref{LocalTheorem}. Differentiating with respect to $x$ the \eqref{MainMFGSystem} or \eqref{PlanningSystem}, we obtain 
\begin{equation}
    \begin{cases}
    - u_{xt}+H_p(u_x,m)u_{xx}+H_m(u_x,m)m_x=0 \text{ in }Q_T,\\
    m_{xt}-(m_x H_p(u_x,m)+mH_{pp}(u_x,m)u_{xx}+mH_{pm}(u_x,m)m_x)_x=0 \text{ in }Q_T.
    \end{cases}
\end{equation}
Testing against $u_x$ in the equation for $m_x$ above we obtain 
\begin{multline}\int_{\T} m_x(T)u_x(T)dx-\int_{\T} m_x(0)u_x(0)dx 
+\int_0^T\int_{\T} m_x(- u_{xt}+u_{xx}H_p(u_x,m))\\+mu_{xx}^2H_{pp}(u_x,m)+mu_{xx}H_{pm}(u_x,m)m_x dx=0, \end{multline}
and, therefore,
\[ \int_{\T} m_x(T)u_x(T)dx +\int_0^T\int_{\T} mu_{xx}^2H_{pp}- H_m(m_x)^2 dx\]
\[=-\int_{\T} u(0)(m_0)_{xx} dx-\int_0^T\int_{\T} mu_{xx}H_{pm}(u_x,m)m_xdx .\]
Next, we use the following bounds
\[\Big|\int_{\T} u(0)(m_0)_{xx} dx \Big|\leq \|u\|_{\infty}\|(m_0)_{xx}\|_1,\]
and, for $\delta \in (0,1)$,
\[\Big| mu_{xx}H_{pm}m_x\Big| \leq (1-\delta)mu_{xx}^2H_pp+\frac{1}{4(1-\delta)}m|H_{pm}|^2(m_x)^2 \]
\[\leq (1-\delta)mu_{xx}^2H_pp-\frac{4}{4(1-\delta)(1+\frac{1}{C_0})}H_m(m_x)^2,\]
where in the last inequality we used \eqref{eq:ellipticity}. Choose $\delta>0$ small enough so that 
\[\frac{1}{(1-\delta)(1+\frac{1}{C_0})}<1.\]
Hence, in the case of (\ref{MainMFGSystem}), we have the bound
\[\intT g'(m(T))(m_x(T))^2dx +\int_0^T\intT mH_{pp}(u_{xx})^2 dx-H_m(m_x)^2 dx\leq C+\|u\|_{\infty}\|(m_0)_{xx}\|_1 \]
while in the case of (\ref{PlanningSystem}), we have
\[\int_0^T\intT mH_{pp}(u_{xx})^2 dx-H_m(m_x)^2 dx\leq C+\|u\|_{\infty}\Big(\|(m_0)_{xx}\|_1+\|(m_T)_{xx}\|_1\Big). \]\end{proof}

\section{Long time behavior and the infinite horizon problem} \label{LongSection}
In this section, we will characterize the behavior, as $T\rightarrow \infty$, of solutions to \eqref{MainMFGSystem} and \eqref{PlanningSystem}. First, we establish the turnpike property with an exponential rate of convergence. This property shows that, for large values of $T$, the players spend most of their time close to the equilibrium $m \equiv 1$. 

\begin{lem} \label{lem: turnpike} Let $(u,m)$ be a solution to \eqref{MainMFGSystem} or \eqref{PlanningSystem}, let $T>1$, and set 
\[c_1=\min(\min m_0,\min m_T),\,\,C_1=\max(\max m_0,\max(m_T)).\] Then there exist constants $C,\omega>0$, with \[C=C(C_0,C_1,c_1^{-1},\|\overline{C}\|_{L^{\infty}([c_1,C_1])},\|(m_0)_x\|_{\infty},\|(m_T)_x\|_{\infty},\|(g')\|_{L^{\infty}([\min m_0,\max m_0])}^{-(\gamma-1)})\,\]
and
\[\omega^{-1}=\omega^{-1}(C_0,c_1^{-1},C_1,\|\overline{C}\|_{L^{\infty}([c_1,C_1])}),\]
such that
\begin{equation} \label{eq:turnpike 1}\|m(t)-1\|_{L^\infty(\T)} + \|u_x(t)\|_{L^\infty(\T)}\leq C(e^{-\omega t}+e^{-\omega(T-t)}), \,\,\, t\in [0,T].\end{equation}
If $(u,m)$ solves $\eqref{MainMFGSystem}$, and $\eqref{RadialConditionH}$ holds, we have
\begin{equation} \label{eq:turnpike 2} \|m(t)-1\|_{L^\infty(\T)} + \|u_x(t)\|_{L^\infty(\T)}\leq Ce^{-\omega t}, \,\,\, t\in [0,T].\end{equation}  \end{lem}
\begin{proof}
As in previous arguments, we recall that the constant $C$ may increase at each step. For each $k\in \N$, Proposition \ref{DisplacementProp} yields
\begin{equation} \label{eq: (m-1)^k convex}\frac{d^2}{dt^2}\intT (m-1)^{2k}dx\geq 0,\end{equation}
and, as a result of \hyperref[eq:H long time]{(L)}, \eqref{eq:Hpp bd}, and Corollary \ref{DisplacementConvexityCorBounds},
\[\frac{d^2}{dt^2}\intT (m-1)^{2}dx\geq\frac{1}{C}\intT -mH_mmH_{pp}m_x^2 dx\geq\frac{1}{C}\intT \left|(m-1)_x\right|^2 dx.\]
Since $\intT m(\cdot,t)\equiv1$, arguing in the same way as in \eqref{eq: FTC appli}, we obtain
\[\frac{d^2}{dt^2}\intT (m-1)^{2}dx\geq \frac{1}{C}\Big\| m-1\Big\|_{\infty}^{2}.\]
Therefore, setting 
\[\phi(t):=\intT (m(t)-1)^{2}dx,\]
we have
\begin{equation}\label{eq:ODE phi}
    -\phi''+\frac{1}{C} \phi \leq 0.
\end{equation}
Moreover, if $(u,m)$ solves \eqref{MainMFGSystem} and \eqref{RadialConditionH} holds, up to increasing the value of $C$, Lemma \eqref{DecreasingLemmaForDisplacement} implies that
\begin{equation}\label{eq:Robin phi}\phi'(T)\leq -\frac{1}{\sqrt{C}}\phi(T).\end{equation}
We now fix the choice $\omega=\frac{1}{2\sqrt{C}}$ (the value of $C$ may still increase in subsequent steps, but the value of $\omega$ will not). The comparison principle applied to \eqref{eq:ODE phi} then implies that, for each $t\in [0,T]$,
\begin{equation} \label{eq:phi2}\phi(t)\leq \phi (0) e^{-2\omega t}+\phi(T)e^{-2\omega(T-t)}\leq C(e^{-2\omega t}+e^{-2\omega(T-t)}). \end{equation}
Similarly, if $(u,m)$ solves \eqref{MainMFGSystem} and \eqref{RadialConditionH}, then  \eqref{eq:ODE phi}, coupled with the Robin boundary condition \eqref{eq:Robin phi}, readily implies that
\begin{equation} \label{eq:phi3}\phi(t)\leq \phi(0)e^{-2\omega t}\leq Ce^{-2\omega t}. \end{equation}
By using the same convexity arguments as in \eqref{eq: convexity argu}, in view of \eqref{eq: (m-1)^k convex}, we have
\begin{equation} \label{eq:m-1 bound 2}\|m(t)-1\|_{\infty}^2\leq C\int_{t-\frac{1}{2}}^{t+\frac{1}{2}}\|m(s)-1\|_{\infty}(s)^2ds\leq C\int_{t-1}^{t+1}\intT (m-1)^2=C\int_{t-1}^{t+1} \phi(s)ds. \end{equation}
We now turn our attention to estimating $u_x$. Fixing $t\in [1,T-1]$, as a result of \eqref{eq:Hpp bd}, Proposition \ref{DisplacementProp}, and Corollary \ref{DisplacementConvexityCorBounds}, we obtain, for $s\in [t-1,t+1]$,
\[\frac{1}{C}\intT u_{xx}^2(s)\leq \frac{d^2}{ds^2}\intT (m(s)-1)^2.\]
Thus, testing against a bump function $\zeta\geq 0$, which is supported on $[t-1,t+1]$, and identically equals $1$ on $[t-\frac{1}{2},t+\frac{1}{2}]$, we get
\begin{equation} \label{eq:uxx inter}\int_{t-\frac{1}{2}}^{t+\frac{1}{2}}\intT u_{xx}^2\leq C\int_{t-1}^{t+1}\intT (m-1)^2\zeta'' \leq C\int_{t-1}^{t+1}\phi(s)ds.\end{equation}
Differentiating \eqref{eq:quasilinear} with respect to $x$, one sees that $v=u_x$ solves a linear elliptic equation of the form
\[-\text{Tr}(A(x,t)D^2v)+b(x,t)\cdot Dv=0.\]
Thus, $v$ satisfies the maximum and minimum principles on compact subsets of $\overline{Q}_T$. Applying this observation to $\T \times [t-s,t+s]$, for $s\in (0,\frac{1}{2})$, as well as the fact that, for every $t\in[0,T]$, $\{x\in \T : u_x(x,t)=0\}\neq \emptyset$, we have
\[\text{osc}_{\T}v(t)\leq \text{osc}_{\T}v(t+s)+\text{osc}_{\T}v(t-s)\leq \intT |u_{xx}(t+s)|+\intT |u_{xx}(t-s)|. \]
Integrating in $s$ then yields
\[\text{osc}_{\T}u_x(t)\leq \int_{t-\frac{1}{2}}^{t+\frac{1}{2}}\intT|u_{xx}|,\]
and, thus, as a result of \eqref{eq:uxx inter} and the Cauchy-Schwarz inequality, 
\begin{equation}\label{eq:ux turnpike bd}\|u_x(t)\|_{\infty}^2\leq C\int_{t-1}^{t+1}\phi(s)ds.\end{equation}
Now, adding \eqref{eq:m-1 bound 2} and \eqref{eq:ux turnpike bd}, followed by \eqref{eq:phi2}, we obtain \eqref{eq:turnpike 1} for $t\in [1,T-1]$. Similarly,  when $(u,m)$ solves \eqref{MainMFGSystem} and \eqref{RadialConditionH} holds, \eqref{eq:phi3} yields \eqref{eq:turnpike 2} for $t\in [1,T-1]$. We observe that, for $t\in [0,T] \backslash [1,T-1]$, the bounds on $\|m(t)-1\|_{\infty}$ given by \eqref{eq:turnpike 1} and \eqref{eq:turnpike 2} hold trivially, up to increasing the value of $C$. Let us see that the same is true for the bounds on $\|u_x(t)\|_{\infty}$ on the interval $[0,1]$. Indeed, we may simply follow the proof of Proposition \ref{GradientEstimateProp}, applied to  the MFG system on the domain $\T \times [0,1]$, with the only change being on Case 1 of that proof, that is, when the maximum value is attained at $t=1$. For this case, we may simply use the fact that, as a result of \eqref{eq:turnpike 1} holding for $t=1$, $|u_x(\cdot,1)|$ is bounded. Thus, if we take $T=1$ in Proposition \ref{OscillationProp}, this yields a bound on $\|u_x\|_{\T \times [0,1]}$ that depends only on $C_0,\|m\|_{L^{\infty}(\overline{Q}_T)},\|m^{-1}\|_{L^{\infty}(\overline{Q}_T)},\|(m_0)_x\|_{\infty},$ and $\|\overline{C}\|_{L^{\infty}([\min m,\max m])}$. A similar argument may be followed on $\T \times [T-1,T]$, which completes the proof. \end{proof}

Having established the turnpike property, we now follow the program developed in \cite{CirantPor} to study the long time behavior. In order to characterize the limit, as $T\rightarrow \infty$, of the functions $(u(t)-\lambda (T-t),m(t))$, we first show a uniqueness result for \eqref{InfiniteHorizonSystem}.

\begin{lem} \label{lem: uniqueness lemma} Assume that \hyperref[eq:H long time]{(L)} holds. Then, up to adding a constant to $v$, there exists at most one classical solution $(v,\mu)$ to \eqref{InfiniteHorizonSystem} satisfying \eqref{(v,mu) space}.\end{lem}
\begin{proof} Assume that $(v^1,\mu^1),(v^2,\mu^2)$ are solutions to \eqref{InfiniteHorizonSystem} satisfying \eqref{(v,mu) space}. Since $\mu^1-1,\mu^2-1\in L^1(\T \times (0,\infty))$, there exists a sequence $T_k\rightarrow \infty$ such that
\[\lim_{k\rightarrow \infty} \intT \left(|\mu^1(\cdot,T_k)-1|+|\mu^2(\cdot,T_k)-1|\right) = 0.\]
Performing the standard Lasry-Lions computation for $v^1,v^2$ on $Q_{T_k}$, using Lemma \ref{EstimateOfHForLasryLions}, and noting that \[\mu^i,(\mu^i)^{-1},v^i_x, \in L^{\infty}(\T \times (0,\infty)),\,\,\,\, i\in\{1,2\},\] we obtain
\begin{multline}\frac{1}{C}\left(\int_0^{T_k} \intT |v^1_x-v^2_x|^2+ |\mu^1-\mu^2|^2\right) \leq \intT -(v^1(T_k)-v^2(T_k))(\mu^1(T_k)-\mu^2(T_k))\\=\intT -(v^1(T_k)-v^2(T_k))((\mu^1(T_k)-1)-(\mu^2(T_k)-1)).\end{multline}
Now, since $v^1,v^2 \in L^{\infty}(\T \times (0,\infty))$, the right hand side converges to $0$ as $k\rightarrow \infty$. Therefore,
\[\int_0^{\infty} \intT |v^1_x-v^2_x|^2+ |\mu^1-\mu^2|^2=0.\]This implies that $\mu^1=\mu^2$ and $v^1_x=v^2_x$. From the HJ equations, $v^1_t=v^2_t$, which concludes the proof. \end{proof}
In the following lemma, we obtain uniform estimates for the solution that are independent of $T$.
\begin{lem}\label{lem: MainEstimatesForLongTimeThm}
Let $(u^T,m^T)$ be a solution to \eqref{MainMFGSystem} or \eqref{PlanningSystem} for $T>0$, and let $\omega>0$ be the constant from Lemma \ref{lem: turnpike}. Set $v^T=u^T-\lambda (T-t)$. Then there exists a constant $C>0$, independent of $T$, such that:
\begin{itemize}
    \item If \eqref{RadialConditionH} holds and $(u^T,m^T)$ solves \eqref{MainMFGSystem}, then
    \begin{equation}\label{eq:ConstantEstimateMFG}
        | v^T(t) - g(1)| \leq Ce^{-\omega t} \text{ for all }t\in [0,T]. 
    \end{equation}
    \item If $(u^T,m^T)$ solves \eqref{PlanningSystem}, and 
    \begin{equation}\label{eq:NormalizationMFGPT/2}
        \int_{\T} v^T\left(\frac{1}{2}T\right)dx =0,
    \end{equation}
    then we have
    \begin{equation}\label{eq:TimeBoundForMFGPMiddleTimes}
        \|v^T\|_{L^{\infty}(Q_T)}\leq C
    \end{equation}
    and
    \begin{equation}\label{eq:ConstantEstimateBoundMFGP}
        \|v^T\|_{L^{\infty}(Q_{\frac{T}{2}})}\leq Ce^{-\omega t}.
    \end{equation}
\end{itemize}
\end{lem}
\begin{proof}
First we note that in both \eqref{MainMFGSystem} and \eqref{PlanningSystem}, as a result of Lemma \ref{lem: turnpike}, the function $v^T_x=u^T_x$ is bounded uniformly, independently of $T$, and, by Corollary \ref{DisplacementConvexityCorBounds}, so are $m^T,(m^T)^{-1}$. Therefore, since $H$ is smooth, and thus locally Lipschitz, we have, for some constant $C>0$ independent of $T>0$,
\begin{equation}\label{eq:TimeDerivativeBoundHJForvT}
    |v_t^T|\leq C(|v_x^T|+|m^T-1|).
\end{equation}
Assume first that $(u^T,m^T)$ solves \eqref{MainMFGSystem} and \eqref{RadialConditionH} holds. Integrating the HJ equation in $[t,T]$ and using \eqref{eq:TimeDerivativeBoundHJForvT} along with \eqref{eq:turnpike 2} in Lemma \ref{lem: turnpike} we obtain
\[|v^T(t)-v^T(T)|\leq C\int_t^T e^{-\omega s}ds.\]
Furthermore, using the fact that 
\[v^T(T)=u^T(T)=g(m^T(T)),\]
and 
\[|m^T(T)-1|\leq Ce^{-\omega T},\]
by increasing the constant $C$ if necessary, we obtain
\[ |v^T(t)-g(1)|\leq C(e^{-\omega T}+e^{-\omega t})\leq 2Ce^{-\omega t}, \]
which proves \eqref{eq:ConstantEstimateMFG}.
Next, we assume that $(u^T,m^T)$ solves \eqref{PlanningSystem} and \eqref{eq:NormalizationMFGPT/2} holds. Letting $t<\frac{T}{2}$, and integrating the HJ equation in $[t,\frac{T}{2}]$, we obtain from \eqref{eq:TimeDerivativeBoundHJForvT} and \eqref{eq:turnpike 1} that
\begin{equation} \label{avg bd 2} \Big|\int_{\T} v^T(\cdot,t)\Big|\leq C\int_t^{\frac{T}{2}} e^{-\omega s}+e^{-\omega(T-s)}ds\leq \frac{2C}{\omega}\Big(e^{-\omega t}+e^{-\omega \frac{T}{2}}\Big)\leq \frac{4C}{\omega}e^{-\omega t}.\end{equation}
Similarly, for $t\geq\frac{T}{2}$ integrating the HJ equation in $[\frac{T}{2},t]$ yields
\begin{equation} \label{eq:avg bd 1}\Big|\int_{\T} v^T(\cdot,t)dx \Big|\leq C. \end{equation}
Now, for every $t\in [0,T]$, there exists a point $x_t\in \T$ such that $v^T(x_t,t)=\int_{\T} v^T(\cdot,t)$. Therefore, 
\[ |v^T(x,t)|\leq \text{osc}_{\T}v^T(t)+\Big|\int_{\T} v^T(\cdot,t)\Big|.\]
As a result, in view of \eqref{eq:turnpike 1}, the estimates \eqref{eq:avg bd 1} and \eqref{avg bd 2} yield, respectively, \eqref{eq:TimeBoundForMFGPMiddleTimes} and \eqref{eq:ConstantEstimateBoundMFGP}.
\end{proof}
We are now ready to prove our last result.

\begin{proof} [Proof of Theorem \ref{thm: long time}]
We set
\[v^T=u^T-\lambda (T-t),\]
and show that $v^T$ is is convergent as $T\rightarrow \infty$.

In view of Lemmas \ref{lem: turnpike} and \ref{lem: MainEstimatesForLongTimeThm}, as well as \eqref{eq:TimeDerivativeBoundHJForvT}, we see that $\|v^T\|_{W^{1,\infty}(Q_T)}$ and $\|m^T\|_{\infty}$ are bounded, independently of $T$. We may therefore apply the Arzelà--Ascoli theorem to conclude that, up to extracting a subsequence, there exist $v\in W^{1,\infty}(\T \times [0,\infty))$ and $\mu \in L^{\infty} (\T \times [0,\infty))$ such that
\[v^T \rightarrow v \text{ locally uniformly in } \T \times [0,\infty),\]
and
\[m^T \rightharpoonup \mu \text{ weakly--* in } L^{\infty}(\T \times (0,\infty)).\]
We now fix $T_0\in (1,\infty)$, and assume that $T>T_0+1$. Then $(v^T,m^T)$ solves the system
\begin{equation}
    \begin{cases}
    -v^T_t+\lambda+H(v^T_x,m^T)=0 & \text{ in }Q_{T_0},\\
    m^T_t-(m^TH_p(v^T_x,m^T))_x=0 & \text{ in }Q_{T_0}, \\
    m^T(\cdot,0)=m_0.
    \end{cases}
\end{equation}
Moreover, as a result of the interior $C^{1,\alpha}$ estimates for quasilinear elliptic equations, and the interior Schauder estimates for linear equations, $m^T(\cdot,T_0)$ is uniformly bounded in $C^{2,\alpha+\epsilon}$, where $\epsilon>0$ is chosen such that $\alpha+\epsilon<1$. Therefore, as in the proof of Theorem \ref{MainTheorem}, we conclude that, as $T\rightarrow \infty$, \begin{equation}
(v^T,m^T)\rightarrow(v,\mu) \text{ in } C^{3,\alpha}(\T  \times [0,T_0])\times C^{2,\alpha}(\T  \times [0,T_0]).\end{equation} 
In particular, this implies that $(v,\mu)\in C^{3,\alpha}_{\text{loc}}(\T \times [0,\infty))\times C^{2,\alpha}_{\text{loc}}(\T \times [0,\infty))$, and that $(v,\mu)$ solves \eqref{InfiniteHorizonSystem}. Letting $T\rightarrow \infty$ in \eqref{eq:turnpike 1} yields
\begin{equation} \label{eq: turnpike limit}
    \|\mu(t)-1\|_{\infty}+ \|v_x(t)\|_{\infty}\leq Ce^{-\omega t},
\end{equation} which shows that $\mu-1 \in L^1(\T \times (0,\infty))$. Moreover, since $\|(m^T)^{-1}\|_{\infty}$ is bounded, we conclude that \eqref{(v,mu) space} holds.

Now, since a subsequence was extracted, we must verify that the limit is uniquely determined. In view of Lemma \ref{lem: uniqueness lemma}, $\mu$ is uniquely determined, and $v$ is uniquely determined up to a constant. In the case of \eqref{MainMFGSystem}  we see from \eqref{eq:ConstantEstimateMFG} that 
\[\lim\limits_{t\rightarrow \infty} \|v(t)-g(1)\|_{\infty} =0.\]
On the other hand, in the case of \eqref{PlanningSystem},  letting $T\rightarrow \infty$ followed by $t\rightarrow \infty$ in \eqref{eq:ConstantEstimateBoundMFGP}, we obtain
\[\lim\limits_{t\rightarrow \infty}\|v(t)\|_{\infty}=0.\]

\end{proof}

\section*{Acknowledgements}
 The authors would like to thank P.E. Souganidis, P. Cardaliaguet, and A. Porretta for valuable discussions and suggestions. The authors were partially supported by P. E. Souganidis' NSF grant DMS-1900599, ONR grant N000141712095 and AFOSR  grant FA9550-18-1-0494.
\bibliographystyle{abbrv} 
\bibliography{references} 
\end{document}